\documentclass{amsart}

\usepackage{graphicx}
\usepackage{amssymb}

\def\R{\mathbb{R}}
\def\Q{\mathbb{Q}}
\def\Z{\mathbb{Z}}
\def\Ker{\mathrm{K}}
\def\K{\mathbb{K}}
\def\C{\mathcal{C}}
\def\N{\mathbb{N}}
\def\Mo{\mathrm{M}}
\def\w{\mathrm{w}}
\def\a{\alpha}
\newcommand{\dsum}{\displaystyle\sum}

\newcommand{\dmax}{\displaystyle\max}

\newcommand{\dlim}{\displaystyle\lim}
\newcommand{\bfrac}[2]{\genfrac{(}{)}{0pt}{}{#1}{#2}}
\usepackage{float}
\usepackage{hyperref}

\newtheorem{thm}{Theorem}[section]
\newtheorem{defn}{Definition}[section]
\newtheorem{lem}{Lemma}[section]
\newtheorem{prop}{Proposition}[section]
\newtheorem{cor}{Corollary}[section]
\newtheorem{ex}{Example}[section]
\newtheorem{rmk}{Remark}[section]

\usepackage{tikz}
\usetikzlibrary{shapes,arrows}
\usepackage{pgfplots}
\pgfplotsset{compat=newest}
\usetikzlibrary{decorations.markings}

\newcommand{\sign}{\mathrm{sgn}}
\usepackage{mathtools}
\usepackage{blkarray} %

\usepackage{xpatch}

\makeatletter
\xpatchcmd{%
\@maketitle}{%
\ifx\@empty\authors \else \@setauthors \fi
}{%
  \ifx\@empty\authors \else \centering\@setauthors \fi
 \ifx\@empty\addresses \else\@setaddresses\fi
}{\typeout{Patch successful}}{\typeout{Patch failed}}
\makeatother

\title{On $\ell_{\lowercase{p}}$-Support Vector Machines and Multidimensional Kernels}

\author{V\'ictor Blanco\textsuperscript{$\dagger$}}
\address{\textsuperscript{$\dagger$}Dpt. Quantitative Methods for Economics \& Business, Universidad de Granada}
\email{vblanco@ugr.es}

\author{Justo Puerto\textsuperscript{$\ddagger$}}
\address{\textsuperscript{$\ddagger$}Dpt. Statistics \& OR, Universidad de Sevilla}
\email{puerto@us.es}

\author{Antonio M. Rodr\'iguez-Ch\'ia\textsuperscript{$\star$}}
\address{\textsuperscript{$\star$}Dpt. Statistics \& OR, Universidad de C\'adiz}
\email{antonio.rodriguezchia@uca.es}

\begin{document}

\begin{abstract}
In this paper, we extend the  methodology developed for Support Vector Machines (SVM) using $\ell_2$-norm  ($\ell_2$-SVM)  to the more general case of  $\ell_p$-norms with $p\ge 1$
($\ell_p$-SVM).  The resulting primal and dual problems are formulated as mathematical
programming problems;  namely, in the primal case, as a  second order cone optimization
problem and in the dual case, as a polynomial optimization problem involving homogeneous
polynomials.  Scalability of the primal problem is obtained via general transformations based
on the expansion of functionals in Schauder spaces. The concept of Kernel function, widely
applied in $\ell_2$-SVM,  is extended to the more general case by defining a new operator
called multidimensional Kernel. This object gives rise to reformulations of dual problems,
in a transformed space of the original data, which are solved by a moment-sdp based approach. The results of some computational experiments on real-world
datasets are  presented showing rather good behavior in terms of standard
indicators such a \textit{accuracy index} and its ability to classify new data.
\end{abstract}



\maketitle

\section{Introduction}

In  supervised classification, given a  finite set of objects partitioned into classes, the goal is to build a mechanism, based on current available information, for classifying new objects  into  these classes.  Due to their successful applications in the last decades, as for instance in writing  recognition~\cite{writing}, insurance companies (to determine whether an
applicant is a high insurance risk or not)~\cite{insurance}, banks (to decide whether an applicant is a good credit risk or not)~\cite{credit}, medicine (to determine whether a tumor is benigne or maligne)~\cite{cleveland,cancer}, etc; support vector machines (SVMs) have become a popular methodology for
supervised classification~\cite{Burges98}.

Support vector machine (SVM) is a mathematical programming tool, originally developed by  Vapnik \cite{Vapnik95,Vapnik98} and Cortes and Vapnik \cite{CortesVapnik95}, which consists in finding  a hyperplane  to separate a set of data into two classes, so that the distance from the hyperplane to the nearest point of each  class is maximized. In order to do that, the standard SVM solves an optimization problem that accounts for both the training error and the model complexity. Thus, if the separating hyperplane is  given as $ \mathcal{H}=\{z \in \R^d: \omega^t z+b=0\}$, the function to be  minimized  is of the form $\frac{1}{2}\|\omega\|^2+C \cdot R_{emp}(\mathcal{H})$, where $\|\cdot\|$ is a norm and $R_{emp}$ is an empirical measure of the risk incurred using the hyperplane $\mathcal{H}$ to classify  the training data. The most popular  version of  SVM is the one using the Euclidean norm to measure the distance. This approach allows for the use of a kernel function as a way of embedding the original data in a higher dimension space where the separation may be easier without increasing the difficulty of solving the problem (the so-called \textit{kernel trick}).

After a fruitful development of the above approach, some years later, it was observed by several authors that the model complexity could be controlled by other norms different from the Euclidean one (see \cite{BennettBredensteiner00,GonzalezAbril11,IkedaMurata05a,IkedaMurata05b,PedrosoMurata01}). Among many other facts, it is  well-known that using $\ell_1$ or $\ell_{\infty}$ norms (as well as any other polyhedral norms) gives rise to SVM  whose induced optimization problems are linear rather than quadratic, making, in principle, possible solving larger size instances. Moreover, it is also agreed that $\ell_1$-SVM tends to generate sparse classifiers that can be more easily interpreted and reduce the risk of overfitting. The use of more general norms, as the family of $\ell_p$, $1<p<+\infty$, has been also  partially investigated  \cite{cr13, GonzalezAbril11,Liu07}. For this later case, some geometrical intuition on the underlying problems has been given but very few is known about the optimization problems (primal and dual approaches), transformation of data, extensions of the kernel tools (that have been extremely useful in the Euclidean case) and about actual applications to classify databases.

The goal of this paper is to develop a common framework for the analysis of $\ell_p$-norm Support Vector Machine ($\ell_{p}$-SVM) with general   $p \in \Q$ and $p> 1$. We shall develop the theory to understand primal and dual versions of this problem using  these norms. In addition, we also extend the concept of kernel as a way of considering data embedded in a  higher dimension space without increasing the difficulty of  tackling the problem,  that in the general case always appears via homogeneous  polynomials and linear functions.
In our approach, we reduce all the primal problems to efficiently solvable Second Order Cone  Programming  (SOCP) problems. The respective dual problems are reduced to solving polynomial optimization problems. A thorough geometrical analysis of those problems allows  for an extension of the kernel trick, applicable to the Euclidean case, to the more general $\ell_p$-SVM. For that extension, we introduce the concept of multidimensional kernel, a mathematical object that makes the above mentioned job.
In addition, we  derive the relationship between multidimensional kernel  functions and real tensors. In particular, we provide sufficient conditions to test whether a symmetric real tensor, of adequate dimension and order, induces one of the above mentioned multidimensional kernel functions. Also, we develop two different approaches to find the separating hyperplanes. The first one is based on the Theory of Moments and it constructs a sequence of SemiDefinite Programs (SDP) that converges to the optimal solution of the problem. The second one uses limited expansions of functionals representable in Schauder spaces~\cite{LiTza77}, and it allows us to approximate any transformation (whose functional belongs to  a Schauder space) in the original space without mapping the data.  Both approaches permit to reproduce the \textit{kernel trick} even without specifying any a priori transformation.  We report the results of an extensive battery of  computational experiments, on  some common real-world instances on the SVM field, which are comparable or superior to those previously known in the literature.

The rest of the paper is organized as follows. In Section \ref{sec:2}, we introduce  $\ell_p$-support vector machines. We derive primal and dual formulations for the problem, as a second order cone programming problem and as polynomial optimization problem involving homogeneous polynomials, respectively. In addition,  using the dual formulation, we give explicit expressions of the separating hyperplanes expressed as homogeneous polynomials  on the original data. In Section \ref{sec:3}, the concept of multidimensional Kernel is defined to extend the Kernel theory for $\ell_2$-SVM to a more general case of $\ell_p$-SVM with $p>1$.
A hierarchy of Semidefinite Programs (SDP) that converges to the actual solution of $\ell_p$-SVM is developed in Section \ref{sec:4}. Finally, in Section \ref{sec:5}, the results of some computational experiments on real-world datasets are reported.

\section{$\ell_p$-norm Support Vector Machines\label{sec:2}}

For a given $p \in \Q$ with $p>1$, the goal of this section is to provide a general framework to deal with  $\ell_p$-SVM.  In $\ell_p$-SVM, the problem will be formulated as a mathematical programming problem whose objective function depends  on the $\ell_p$-norm of some  of the decision variables.
The input data for this problem is a set of $d$ quantitative measures about $n$ individuals. The $d$ measures about each individual $i \in \{1, \ldots, n\}$ are identified with the vector $\mathrm{x}_{i \cdot} \in \R^d$, while for $j\in \{1, \ldots, d\}$, the $n$ observations about the $j$-th measure are denoted by $\mathrm{x}_{\cdot j} \in \R^n$. The $i$th individual is also classified into a class  $y_i$,  with $y_i \in \{-1, 1\}$, for $i=1,\ldots, n$. The  classification pattern is defined  by $\mathbf{y}=(y_1, \ldots, y_n) \in \{-1,1\}^n$.\\
The goal of SVM is to find a hyperplane $\mathcal{H} = \{z \in \R^d: \omega^t z + b =0\}$ in $\mathbb{R}^d$ that minimizes the misclassification of data to their own class in the sense that is  explained  below.
 SVM tries to find a band defined by two parallel hyperplanes,
 $\mathcal{H}_+ = \{z \in \R^d: \omega^t  z  + b = 1\}$ and $\mathcal{H}_{-} = \{z \in \R^d: \omega^t  z  + b =-1\}$ of maximal width without misclassified observations. The ultimate aim is that each class belongs to one of the halfspaces determined by the strip. Note that if the data are linearly separable, this constraint can be written as follows:
$$y_i (\omega^t \mathrm{x}_{i\cdot}  + b) \geq 1, \quad i=1, \ldots, n.$$
Since in many cases a linear separator is not possible, misclassification is allowed by adding a variable $\xi_i$ for each individual which will take value $0$ if the observation is adequately classified with  respect to this strip, i.e., the above constraints are fulfilled  for that individual;  and it will take a positive value proportional on how far is the observation  from being well-classified. (This misclassifying error is usually called the \textit{hinge--loss} of the  $i$th individual and represents the amount $\max\{0, 1-y_i(\omega^t \mathbf{x}_{i\cdot}+b)\}$,  for all $i=1, \ldots, n$.) Then, the constraints to be satisfied are:
$$
y_i (\omega^t  \mathrm{x}_{i\cdot}  + b) \geq 1 - \xi_i , \quad \forall i =1, \ldots, n.
$$
Therefore, the goal will be  simultaneously to maximize the margin between the two  hyperplanes, $\mathcal{H}_+$  and $\mathcal{H}_{-}$
and   to minimize  the deviation of  misclassified observations.
To measure the norm-based margin between the hyperplanes $\mathcal{H}_+$ and $\mathcal{H}_{-}$, one can use the results by Mangasarian \cite{mangasarian}, to obtain that whenever the distance measure is the $\ell_q$-norm (with $\frac{1}{p}+\frac{1}{q}=1$), the margin for $\ell_p$-SVM is exactly $\frac{2}{\|\omega\|_p}$ (Recall that $\|\omega\|_q$ is the dual norm of $\|\omega\|_p$.). Henceforth, we assume without loss of generality that $q=\frac{r}{s}>1$, with $r, s \in \mathbb{Z}_+$ and $\gcd(r,s)=1$.

Next, for the deviation of misclassified observations, one can take the summation of the slack variables $\xi_i$ as a measure for   that term  in the objective function.  Thus, the problem of finding the best hyperplane based on the above two criteria can be equivalently modeled with the aggregated objective function $\|\omega\|_p^p +  C\sum_{i=1}^n \xi_i$, where $C$ is a  parameter of the model representing the tradeoff between the margin and the deviation of misclassified points (weighting the importance given to the correct classification of the observations in the training dataset or to the ability of the model to classify out-sample data).

Hence, the  $\ell_p$-SVM problem can be formulated as:
\begin{align}
\rho^* = \min &\: \: \: \|\omega\|_p^p + C \dsum_{i=1}^n \xi_i&\label{svm0}\tag{$\ell_p{\rm  \mbox{-}SVM}$}\\
\mbox{ s.t. }  &\: \: \:  y_i (\omega^t  \mathrm{x}_{i\cdot}  + b) \geq 1 - \xi_i, &\forall i =1, \ldots, n,\nonumber\\
                     &\: \:\: \xi_i \geq 0, &\forall i =1, \ldots, n,\nonumber\\
                     &\: \:\:   \omega \in \R^d, b\in\R.&\nonumber
\end{align}

Observe that the above problem is a convex nonlinear optimization problem which can be efficiently solved using global optimization tools. Actually, it can be formulated as the following  convex optimization problem with a linear objective function, a set of  linear constraints and a single nonlinear inequality constraint:
\begin{align}
\min &\;\;t + C \dsum_{i=1}^n \xi_i&\label{cc1}\\
\mbox{ s.t. }& y_i (\omega^t  \mathrm{x}_{i\cdot}  + b) \geq 1 - \xi_i,& \forall i =1, \ldots, n,\\
& t \geq \|\omega\|_p^p, &\label{c:q}\\
& \xi_i \geq 0, &  \forall i =1, \ldots, n,\\
&  \omega \in \R^d, b,  t \in \R,&\label{cc5}
\end{align}
where constraint $t \geq \|\omega\|_p^p$ can be conveniently reformulated by introducing new variables $v_j$ and  $u_j$  to account for $|\omega_j|$ and $|\omega_j|^p$, respectively (note that  $p=\frac{r}{r-s}$), for $j=1, \ldots, d$:

 \begin{align}
& v_j \geq \omega_j \,  &\forall j=1, \ldots, d,\nonumber\\
& v_j \geq -\omega_j \, &\forall j=1, \ldots, d,\nonumber\\
&t \geq \sum_{j=1}^d u_j,&\nonumber\\
& u_j^{r-s} \geq v_j^r, &\forall j=1, \ldots, d. \label{socprs}
\end{align}

Although the above formulation is still nonlinear, constraints in \eqref{socprs} can be efficiently rewritten as a set of second order cone constraints and then solved via interior point algorithms (see \cite{BEP14}).

 At this point, we would like to remark that the cases $p=1, +\infty$ also fit (with slight simplifications) within the above framework and obviously they both give rise to linear programs that can be solved via standard linear programming tools. For that reason, we do not follow up with  their analysis in this paper that focus on more general problems that fall in the class of conic linear programming,
  i.e., we assume without loss generality that $1<p < +\infty$.

A second reformulation is also possible using its Lagrangian dual formulation.
\begin{prop}
 The Lagrangian dual problem of \eqref{svm0} can be formulated as a polynomial optimization problem.
\end{prop}
\begin{proof}
Observe first that \eqref{svm0} is convex and satisfies Slater's qualification constraint therefore it has zero duality gap with respect to the Lagrangian dual problem. Its Lagrangian function is:
\begin{equation}
\label{lag}\tag{${\rm LD}$}
L(\omega, b; \alpha, \beta) =  \|\omega\|_p^p + C \dsum_{i=1}^n \xi_i - \dsum_{i=1}^n \alpha_i (y_i (\omega^t  \mathrm{x}_{i\cdot}  + b) + \xi_i -1) - \dsum_{i=1}^n \beta_i \xi_i ,
\end{equation}
where $\alpha_i$ is the dual variable associated to constraints $y_i (\omega^t  \mathrm{x}_{i\cdot}  + b) \geq 1 - \xi_i$ and $\beta_i$ the one for constraints $\xi_i\geq 0$, for $i=1, \ldots, n$.
The KKT optimality conditions for the problem read as:
\begin{eqnarray}
\label{c1}
 & & \dfrac{\partial L}{\partial \omega_j} =  p |\omega_j|^{p-1}  \sign(w_j)- \dsum_{i=1}^n \alpha_i y_i x_{ij} =0, \quad \forall j=1,\ldots,d, \\
 \label{c2}
& & \dfrac{\partial L}{\partial b} = \dsum_{i=1}^n \alpha_i y_i =0,\\
\label{c3}
& &\dfrac{\partial L}{\partial 	\xi_i} = C - \alpha_i - \beta_i =0, \quad \forall i=1,\ldots,n,\\
\label{c4}
& &\alpha_i, \beta_i \geq 0, \quad \forall  i=1,\ldots,n.\nonumber
\end{eqnarray}
where $\sign(\cdot)$ stands for the sign function.

Hence, applying conditions \eqref{c2} and \eqref{c3}, we obtain the following alternative expression  of \eqref{lag}:
\begin{equation*}
\label{lag2}
L(\omega, b; \alpha) = \|\omega\|_p^p - \dsum_{i=1}^n \alpha_i y_i \omega^t  \mathrm{x}_{i\cdot}  + \dsum_{i=1}^n \alpha_i.
\end{equation*}
In addition, from \eqref{c1} and taking into account that  $\frac{1}{p-1}=q-1$, we can reconstruct the optimal value of $\omega_j$ for any  $j=1,\ldots, d$, as follows:

$$ |\omega_j|=\frac{1}{p^{q-1}} \Big(\sign(\omega_j)\sum_{i=1}^n \alpha_i y_i x_{ij}\Big)^{q-1}, $$
and then,
 $$  \omega_j=\frac{1}{p^{q-1}} \sign(\omega_j) \Big(\sign(\omega_j)\sum_{i=1}^n \alpha_i y_i x_{ij}\Big)^{q-1}. $$
Observe that the above two expressions are well-defined for any $q \geq 1$, because by  \eqref{c1}, we have that
$\sign(\omega_j)\Big(\sum_{i=1}^n \alpha_i y_i x_{ij} \Big)\ge 0$.

Actually, by \eqref{c1} we have that
\begin{equation}\label{signo}
\sign(\omega_j) = \sign\left(\dsum_{i=1}^n \alpha_i y_i x_{ij}\right)=:\mathcal{S}_{\alpha,j}.
\end{equation}
Hence:
\begin{equation}\label{eq1}
\omega_j=\frac{1}{p^{q-1}} \mathcal{S}_{\alpha,j} \Big(\mathcal{S}_{\alpha,j}\sum_{i=1}^n \alpha_i y_i x_{ij}\Big)^{q-1}.
\end{equation}

Therefore, again the Lagrangian dual  function can be rewritten as:
 \begin{eqnarray*}
L(\alpha) &=&\left(\dfrac{1}{p^{q}}\right) \dsum_{j=1}^d \left(\left|\dsum_{i=1}^n \alpha_i y_i  x_{ij}\right|^{q-1}\right)^p - \left( \dfrac{1}{p^{q-1}}\right) \dsum_{i=1}^n \dsum_{j=1}^d
\alpha_i y_i x_{ij}\mathcal{S}_{\alpha,j} \Big(\mathcal{S}_{\alpha,j}\sum_{k=1}^n \alpha_k y_k x_{kj}\Big)^{q-1}+ \dsum_{i=1}^n \alpha_i\\
&= &
\left(\dfrac{1}{p^{q}}\right) \dsum_{j=1}^d \left|\dsum_{i=1}^n \alpha_i y_i  x_{ij}\right|^{q}
- \left(\dfrac{1}{p^{q-1}}\right)  \dsum_{j=1}^d \left|\dsum_{i=1}^n \alpha_i y_i  x_{ij}\right|^{q} + \dsum_{i=1}^n \alpha_i \\
&=& \left(\dfrac{1}{p^{q}}-\dfrac{1}{p^{q-1}} \right) \dsum_{j=1}^d \left|\dsum_{i=1}^n \alpha_i y_i  x_{ij}\right|^{q} + \dsum_{i=1}^n \alpha_i
\end{eqnarray*}
provided that $\dsum_{i=1}^n \alpha_i y_i =0$ and $0 \leq \alpha_i \leq C$,  $\forall i=1,\ldots,n$.

Thus, the Lagrangian dual problem may be formulated as  follows:
\begin{align}
\max  &\:\: \:  \left(\dfrac{1}{p^{q}} - \dfrac{1}{p^{q-1}}\right) \dsum_{j=1}^d \left|\dsum_{i=1}^n \alpha_i y_i  x_{ij}\right|^q + \dsum_{i=1}^n \alpha_i\label{of:lagrangean}\tag{${\rm P}_{\rm LD}$}\\
\mbox{ s.t. }&\:\: \: \dsum_{i=1}^n \alpha_i y_i =0,&\nonumber\\
&\:\: \: 0 \leq \alpha_i \leq C,\quad \forall i=1,\ldots,n. \nonumber&
\end{align}

Introducing the variables $\delta_j$ and $u_j$ to represent $\left| \dsum_{i=1}^n \alpha_i y_i x_{ij} \right|$ and $\left| \dsum_{i=1}^n \alpha_i y_i x_{ij} \right|^q$, taking into account that the coefficient
 $\left(\dfrac{1}{p^{q}} - \dfrac{1}{p^{q-1}}\right)$ is always negative for $1 < p, q < +\infty$ and considering $q=\frac{r}{s}$, the problem above is equivalent  to the following polynomial optimization problem:
\begin{align}
\max  &\:\: \:  \left(\dfrac{1}{p^{q}} - \dfrac{1}{p^{q-1}}\right) \dsum_{j=1}^d u_j + \dsum_{i=1}^n \alpha_i& &\label{pop:lag},\tag{${\rm POP}_{\rm LD}$}\\
s.t.\;  &\:\: \:  \delta_j \ge \dsum_{i=1}^n \alpha_i y_i x_{ij} , &\; \forall j=1,\ldots,d,\nonumber\\
& \:\: \: \delta_j \ge -\dsum_{i=1}^n \alpha_i y_i x_{ij} ,& \; \forall j=1,\ldots,d\nonumber\\
&\:\: \: u_j^s\ge \delta_j^r,\;  &\forall j=1,\ldots,d,\nonumber\\
&\:\: \: \dsum_{i=1}^n \alpha_i y_i =0,\nonumber\\
&\:\: \: 0 \leq \alpha_i \leq C,& \forall i=1,\ldots,n.\nonumber
\end{align}
\end{proof}

The reader may observe that the problem \eqref{pop:lag} simplifies further for the cases of integer  $q$ ($q=r$ and $s=1$), and especially if $r$ is even, which results in:
\begin{align}
\max  &\:\: \:   \left(\dfrac{1}{p^{q}} - \dfrac{1}{p^{q-1}}\right) \dsum_{j=1}^d  \delta_j + \dsum_{i=1}^n \alpha_i \nonumber \\
\mbox{ s.t. }  &\:\: \:  \delta_j \ge \left( \dsum_{i=1}^n \alpha_i y_i x_{ij} \right)^r, &\forall j=1,\ldots,d, \nonumber\\
&\:\: \: \dsum_{i=1}^n \alpha_i y_i =0,\nonumber \\
&\:\: \: 0 \leq \alpha_i \leq C,&  \forall i=1,\ldots,n. \nonumber
\end{align}
In order to extend the kernel theory developed for $\ell_2$-norm to a general $\ell_p$-norm, now we study an alternative formulation of the Lagrangian dual problem  in terms of linear functions and homogeneous polynomials in $\alpha$. For this analysis we consider the case where   $q=\frac{r}{s}$ with $s=1$. In order to simplify the proposed formulations we denote by ${\mathrm H}_\mathbf{y} = \{\alpha \in [0,C]^n: \sum_{i=1}^n \alpha_iy_i=0\}$, the feasible region  of \eqref{of:lagrangean} where the dual variables $\alpha$ belong to.

\begin{thm}\label{t:12} There exists an arrangement  of hyperplanes of $\mathbb{R}^n$, such that, in each of its full dimensional elements:
\begin{itemize}
\item[\rm i)]
\eqref{of:lagrangean} can be formulated as a mathematical programming problem  with  objective function  given by a linear  term 
plus a  homogeneous polynomial of degree $r$ on $\alpha$ and linear constraints.
 \item[\rm ii)]  \eqref{of:lagrangean}  and the separating hyperplane it induces depend on the original data throughout homogeneous polynomials.
\end{itemize}
\end{thm}
\begin{proof}
Using \eqref{signo}, the first addend of the objective function of  \eqref{of:lagrangean} can be rewritten as:
$$
\dsum_{j=1}^d \left|\dsum_{i=1}^n \alpha_i y_i  x_{ij}\right|^{ r} = \dsum_{j=1}^d \left(\mathcal{S}_{\alpha,j}\dsum_{i=1}^n \alpha_i y_i  x_{ij}\right)^{ r}.
$$
The above is a piecewise multivariate polynomial in $\alpha$   (recall  that  $y$ and $x$  are input data) with a finite number of ``branches"  induced by the different signs of the terms $\mathcal{S}_{\alpha,j}^r$ for all $j=1,\ldots,d$. Each branch is obtained fixing arbitrarily $\mathcal{S}_{\alpha,j}^r$ to +1 or -1. The domains of these branches are defined by the arrangement induced by the set of homogeneous hyperplanes $\left\{\sum_{i=1}^n \alpha_i y_i  x_{ij}=0,\; j=1,\ldots,d\right\}$. It is well-known that this arrangement has $O(2^d)$ full dimensional subdivision elements that we shall call \textit{cells},    see \cite{Edelsbrunner87}; all of them pointed, closed, convex  cones. Since a generic cell is univocally defined by the signs of the expressions $\dsum_{i=1}^n \alpha_i y_i  x_{ij}$ for $j=1, \ldots,d$,   denote by
$\C(\mathrm{s}_1,\ldots,\mathrm{s}_d)=\{\alpha\in \mathbb{R}^n: \mathcal{S}_{\alpha,j}=\mathrm{s}_j ,\; j=1,\ldots,d\}$,
with $\mathrm{s}_j\in \{-1,1\}$ for all $j=1, \ldots,d$. Next, for all $\alpha\in \C(\mathrm{s}_1,\ldots,\mathrm{s}_d)$ the signs are constant and this allows us to  remove the absolute value in the expression of the first addend of the objective function of \eqref{of:lagrangean}   and then to rewrite it as sum of monomials of the same degree. Indeed, denoting by $\mathrm{z}^\gamma  := z_{1}^{\gamma_1} \cdots z_{n}^{\gamma_n}$, for all $z=(z_1,\ldots,z_n) \in \mathbb{R}^n$ and $\gamma= (\gamma_1,\ldots,\gamma_n)\in \mathbb{N}^n$, we have the following  equalities for any $\alpha \in \C(\mathrm{s}_1,\ldots,\mathrm{s}_d)$:
\begin{align*}
 \dsum_{j=1}^d \hspace*{-0.1cm}\left|\dsum_{i=1}^n \alpha_i y_i  x_{ij}\right|^{ r}  \hspace*{-0.2cm}= &
 \hspace*{-0.1cm}\dsum_{j=1}^d \left(\mathrm{s}_{j} \hspace*{-0.15cm}\dsum_{i=1}^n \alpha_i y_i  x_{ij}\right)^{ r} \hspace*{-0.2cm}= \hspace*{-0.1cm}\dsum_{j=1}^d \left(\dsum_{\gamma \in \N^n_r}
 \mathrm{s}_{j}^r c_{\gamma} \alpha^{\gamma} y^{\gamma}   \mathrm{x}_{\cdot j}^{\gamma} \right) \hspace*{-0.15cm}= \hspace*{-0.15cm}\dsum_{\gamma \in \N^n_r}   c_{\gamma} \alpha^{\gamma} y^{\gamma} \hspace*{-0.1cm}\dsum_{j=1}^d  \mathrm{s}_{j}^r \mathrm{x}_{\cdot j}^{\gamma}
\end{align*}
where  $c_{\gamma} =   {\displaystyle {{ \sum_{i=1}^n \gamma_i}\choose{\gamma_1, \ldots, \gamma_n} } }= \dfrac{(\sum_{i=1}^n \gamma_i)!}{\gamma_1! \cdots \gamma_n!}$, and $\N^n_a:= \{\gamma \in \N^n: \sum_{i=1}^n \gamma_i = a\}$, for any $a \in \N$.

The above discussion justifies the validity of the following representation of \eqref{of:lagrangean} within the cone $\C(\mathrm{s}_1,\ldots,\mathrm{s}_d)$:
\begin{align}
\max \; &   \left(\dfrac{1}{p^{r}} - \dfrac{1}{p^{r-1}}\right) \dsum_{\gamma \in \N^n_r} c_{\gamma} \alpha^{\gamma} y^{\gamma} \dsum_{j=1}^d  \mathrm{s}_{j}^r \mathrm{x}_{\cdot j}^{\gamma} + \dsum_{i=1}^n \alpha_i\label{fo}\\
\mbox{s.t. } &  \mathrm{s}_{j} \dsum_{i=1}^n \alpha_i y_i x_{ij} \geq 0, \qquad \forall j=1, \ldots, d,\\
 & \dsum_{i=1}^n \alpha_i y_i =0,\\
 &0 \leq \alpha_i \leq C, \qquad \forall i=1,\ldots,n.\label{fin}
\end{align}

Finally,   let us deduce the expression of the separating hyperplane  as a function of the optimal solution of \eqref{of:lagrangean},  $\bar \alpha$. For a particular $z \in \mathbb{R}^d$ the separating hyperplane is $\mathcal{H} = \{(z_1, \ldots, z_d) \in \R^d: \sum_{j=1}^d \omega_{j} z_j+b=0\}$. Using \eqref{eq1}, this hyperplane is given by:
\begin{eqnarray*}
\nonumber
\sum_{j=1}^d \frac{1}{p^{r-1}} \mathcal{S}_{\bar \alpha,j}^{r} \Big(\sum_{i=1}^n \bar \alpha_i y_i x_{ij}\Big)^{r-1} z_{j}+b&=&0,\\
\end{eqnarray*}
where the signs are those associated to $\bar \alpha$. Equivalently,
$$
\frac{1}{p^{r-1}}
\dsum_{\gamma \in \N^n_{r-1}}
c_{\gamma} \bar \alpha^{\gamma} y^{\gamma} \sum_{j=1}^d \mathcal{S}_{\bar \alpha,j}^{r} \mathrm{x}_{\cdot j}^{\gamma} z_{j}+b=0.
$$
Finally, to compute $b$, for any $i_0 \in \{1, \ldots, n\}$ with $0 < \bar{\alpha}_{i_0} < C$, by the complementary slackness conditions we get that  we can also reconstruct the intercept of the hyperplane:
$$b = y_{i_0} - \dsum_{j=1}^d \bar{\omega}_j x_{i_0j} =  y_{i_0}- \frac{1}{p^{q-1}} \dsum_{j=1}^d \mathcal{S}_{\alpha,j} \Big(\mathcal{S}_{\bar{\alpha},j}\sum_{i=1}^n \bar{\alpha}_i y_i x_{ij}\Big)^{q-1} x_{i_0j},
$$
and the result follows.
\end{proof}
Observe that the even case can be seen as a particular case of the odd case in which a single arrangement is considered whose signs are all equal to one.

\begin{cor}
For even $r$,  the Lagrangian dual  problem, \eqref{of:lagrangean}, is given as:
\begin{equation}
\label{ast}
\max_{\alpha \in {\mathrm H}_\mathbf{y}} \; 
 \left(\dfrac{1}{p^{r}} - \dfrac{1}{p^{r-1}}\right) \dsum_{\gamma \in \N^n_r} c_{\gamma} \alpha^{\gamma} y^{\gamma} \dsum_{j=1}^d  \mathrm{x}_{\cdot j}^{\gamma} + \dsum_{i=1}^n \alpha_i.
\end{equation}
\end{cor}
\begin{proof}
Note that if $r$ is even one has that:
$$
\dsum_{j=1}^d \left|\dsum_{i=1}^n \alpha_i y_i  x_{ij}\right|^{ r} = \dsum_{j=1}^d \left(\dsum_{i=1}^n \alpha_i y_i  x_{ij}\right)^{ r}.
$$
Hence, the arrangement of hyperplanes (and signs patterns) are not needed in this case and the result follows.
\end{proof}

\begin{rmk}\label{rmk1}
Observe that formulation \eqref{fo}-\eqref{fin} can be slightly modified to be valid for the case $q=\frac{r}{s}$ with $s\ne 1$ as follows:
 \begin{align*}%
\max  \:&  \left(\dfrac{1}{p^{q}} - \dfrac{1}{p^{q-1}}\right) \dsum_{j=1}^d \delta_j + \dsum_{i=1}^n \alpha_i\nonumber\\
\mbox{s.t. } \: &  \dsum_{\gamma \in \N_r^n}  c_{\gamma} \alpha^{\gamma} y^{\gamma} \mathrm{s}_{j}^r  \mathrm{x}_{\cdot j}^{\gamma}-\delta_j^{s}  \le 0,   &   \forall j=1,\ldots,d,\nonumber\\
        &  \mathrm{s}_{j} \dsum_{i=1}^n \alpha_i y_i x_{ij} \geq 0, &\forall j=1, \ldots, d, \nonumber\\
        & \dsum_{i=1}^n \alpha_i y_i =0,   \nonumber\\
        & 0 \le \alpha_i \le C,  &\forall i=1,\ldots,n, \nonumber\\
        &  \delta_i \ge 0, &\forall j=1,\ldots,d. \nonumber
\end{align*}
\end{rmk}

 The results in  Theorem \ref{t:12}, namely that  the Lagrangian problem \eqref{of:lagrangean} and the separating hyperplane it induces depend on the original data throughout homogeneous polynomials;  is the basis to introduce the concept of multidimensional kernel that extends further the kernel trick already known for the SVM problem with Euclidean distance.  This is the aim of the following section.

\section{Multidimensional Kernels}\label{sec:3}
As mentioned in the introduction, when the linear separation between two sets is not clear in the original space, a common technique in supervised classification is to embed the data in a space of higher dimension where this separation may be easier. If we consider
$\Phi: \R^{d} \rightarrow \R^{D}$, a transformation on the original data, the expressions of the Lagrangian dual problem and the separating hyperplane of these transformed data would depend on the  function $\Phi$. In this sense, the increase of the dimension of the space  would be translated in  an increase of the difficulty to tackle the resulting problem. However, when the $\ell_2$-norm is used, the so-called \emph{kernel trick} provides expressions of the Lagrangian dual problem and the separating hyperplane that just depend on the so-called Kernel function.
Basically, the idea behind the kernel trick is to use a Kernel function to handle transformations on the data, and incorporate them to the SVM problem, without the explicit knowledge of the transformation function.
Therefore, although  implicitly we are solving a problem in a higher dimension, the resulting problem is stated in the dimension of the original data and as a consequence, it has  the same difficulty than the original one. Our goal in this section is to extend this idea of the kernel trick to  $\ell_p$-SVM.  In order to do that,
consider a data set $[\mathbf{x}] = (x_{1\cdot}, \ldots, x_{n\cdot})$ together with their classification pattern $\mathbf{y}=(y_1, \ldots, y_n)$ and $r \in \N$.
Given $\Phi: \R^{d} \rightarrow \R^{D}$, the set-valued function ${\rm S}_{\Phi}: 2^{\mathrm{H}_\mathbf{y}} \rightarrow 2^{\{-1,1\}^D}$ ($2^{\mathrm{H}_\mathbf{y}}$ stands  for the power set of  $\mathrm{H}_\mathbf{y}$), is defined as:
\begin{align*}
 {\rm S}_{\Phi}(R):= \Big\{&\mathrm{s} \in \{-1,1\}^D: \mathrm{s}_j =\sign\left(\dsum_{i=1}^n \alpha_i y_i \Phi_j(\mathrm{x}_{i\cdot})\right)^r,  \\
  &\mbox{ for } j=1, \ldots, D, \mbox{ for some } \alpha \in R\Big\}.
\end{align*}

In what follows, we say that  the family of sets  $\{R_k\}_{k \in \mathcal{K}} \subseteq 2^{\mathrm{H}_\mathbf{y}}$  is  a \emph{subdivision} of  $\mathrm{H}_{\mathbf{y}}$  if: (1) $\mathcal{K}$ is finite; and (2) $\bigcup_{k \in \mathcal{K}} R_k=\mathrm{H}_{\mathbf{y}}$ and
$\mathrm{ri}(R_k) \cap \mathrm{ri}(R_{k'})=\emptyset$ for any $k,k' (k\ne k') \in \mathcal{K}$ (where $\mathrm{ri}(R)$ stands for the relative interior of a set $R$).

\begin{defn}
Given a transformation function, $\Phi: \R^{d} \rightarrow \R^{D}$,
a subdivision $\{R_k\}_{k \in \mathcal{K}}$  is said a \emph{suitable $\Phi$-subdivision} of $\mathrm{H}_{\mathbf{y}}$ if
$$
{\rm S}_{\Phi}(R_k) = \{\mathrm{s}_{R_k}\} \text{ for some } \mathrm{s}_{R_k}\in \{-1,1\}^D \text{ and for all } k \in \mathcal{K}.
$$
\end{defn}

Observe that  the signs of $\dsum_{i=1}^n \alpha_i y_i \Phi_j(\mathrm{x}_{i\cdot})$, for $j=1, \ldots, D$, are constant within any element $R_k$ of a suitable  $\Phi$-subdivision. Hence, any finer subdivision of a suitable  $\Phi$-subdivision remains suitable.
Also, one may construct the maximal subdivision of $\mathrm{H}_{\mathbf{y}}$ with such a property by defining:
$$
\C(\mathrm{s}_1, \ldots, \mathrm{s}_D) = \Big\{\alpha \in {\mathrm H}_\mathbf{y}: \sign\Big( \sum_{i=1}^n \alpha_i y_i \Phi_j(\mathrm{x}_{i\cdot}) \Big)^r = \mathrm{s}_{j}, \mbox{ for } j=1, \ldots, D\Big\}
$$
for any $\mathrm{s} \in \{-1,1\}^D$, and choosing  $\{R_k\}_{k\in \mathcal{K}} = \Big\{\C(\mathrm{s}_1, \ldots, \mathrm{s}_D)\Big\}_{\mathrm{s} \in \{-1,1\}^D}$  (observe that  each set of  this subdivision is defined
univocally by a vector $\mathrm{s} \in \{-1,1\}^D$).

\begin{defn}
Given a suitable  $\Phi$-subdivision, $\{R_k\}_{k \in \mathcal{K}} \subseteq 2^{\mathrm{H}_\mathbf{y}}$, and $(\gamma,\lambda) \in \N_r^{n+1}$,
 $\lambda \in \{0,1\}$, the operator
 \begin{equation}\label{kernel0}
\Ker[\mathbf{x}]_{R_k,\gamma,\lambda} (z):= \sum_{j=1}^D
\mathrm{s}_{R_k,j}^r \Phi_j(\mathrm{x})^{\gamma} \Phi_j(z)^{\lambda}, \forall z \in \R^d,  \forall k \in \mathcal{K},
\end{equation}
is  called a $r$-order Kernel  function of $\Phi$.   For $k\in \mathcal{K}$, $\Ker[\mathbf{x}]_{R_k,\gamma,\lambda} (z)$ is called the $k$-th slice of the kernel function. 
\end{defn}

The reader can observe that the objective  function of \eqref{of:lagrangean}  and the separating hyperplane 
obtained as a result of solving this problem can be rewritten for the $\Phi$-transformed data using the Kernel function \eqref{kernel0}.

Indeed, using \eqref{fo} the objective function of the  Lagrangian dual problem when using $\Phi(\mathrm{x})$ instead of $\mathrm{x}$ is:
\begin{eqnarray*}
 \left(\dfrac{1}{p^{r}} - \dfrac{1}{p^{r-1}}\right) \dsum_{\gamma \in \N^n_r} c_{\gamma} \alpha^{\gamma} y^{\gamma}   \Ker[\mathbf{x}]_{R_k,\gamma,0} (z)+ \sum_{i=1}^n \alpha_i,   \quad \forall \alpha \in R_k, \:
  \forall k \in \mathcal{K}.
 \end{eqnarray*}
Since the separating hyperplane is built for  $\alpha^* \in R_{k^*}$, the optimal solution of an optimization problem,  the expression \eqref{ast} of this hyperplane is given by:
$$
 \frac{1}{p^{r-1}}
\dsum_{\gamma \in \N_{r-1}^{n}} c_{\gamma} \alpha^{* \gamma} y^{\gamma} \Ker[\mathbf{x}]_{R_{k^*},\gamma,1} (z)+b=0,\qquad  \; \mbox{ for }  k^*\in \mathcal{K} \mbox{ such that } \alpha^* \in R_{k^*}.
$$

\begin{rmk}

The general definition of kernel simplifies whenever $r$ is even. In such a case, the sign coefficients are no longer needed. Hence, $\{\mathrm{H}_\mathbf{y}\}$ (with $|\mathcal{K}|=1$) is a suitable $\Phi$-subdivision of $\mathrm{H}_\mathbf{y}$
for any transformation $\Phi: \R^d \rightarrow \R^D$. Then, the kernel function becomes:
$$
\Ker[\mathbf{x}]_{\mathrm{H}_\mathbf{y},\gamma,\lambda} (z):= \sum_{j=1}^D
\Phi_j(\mathrm{x})^{\gamma} \Phi_j(z)^{\lambda}, \qquad \forall z \in \R^d,
$$
for $(\gamma,\lambda) \in \N_r^{n}$ and $\lambda \in \{0,1\}$, but being it independent of $\alpha$ (since ${\rm S}_{\Phi}(\mathrm{H}_\mathbf{y}) =\{(1,\stackrel{D}{\ldots},1)\}$).
\end{rmk}

\begin{rmk}
For the Euclidean case ($r=2$), note that usual definition of kernel is $K(z, z^\prime) = \Phi(z)^t \Phi(z^\prime)$ which is independent of the observations. Nevertheless, such an expression is only partially exploited in its application to the SVM problem. For solving the dual problem, $K$ is
applied to pairs of observations, i.e., only through $K( \mathrm{x}_{i_1\cdot},\mathrm{x}_{i_2\cdot})$ for $i_1, i_2=1, \ldots, n$, whereas for classifying an arbitrary  observation $z$, the unique expressions to be evaluated are of the form $K(\mathrm{x}_{i\cdot},z)$.

Thus,  the kernel for the Euclidean case can be expressed:
$$
 K(\mathrm{x}_{i_1\cdot}, \mathrm{x}_{i_2\cdot})=\Phi( \mathrm{x}_{i_1\cdot})^t  \Phi(\mathrm{x}_{i_2\cdot}) =\Ker[\mathbf{x}]_{\mathrm{H}_\mathbf{y},\gamma,0} (z) , \quad \forall z\in \R^d
$$
for $(\gamma,\lambda)={\rm e}_{i_1} + {\rm e}_{i_2}$, $i_1, i_2=1, \ldots, n$, with $\lambda=0$, and
$$
 K(\mathrm{x}_{i_1\cdot}, z)=\Phi(\mathrm{x}_{i_1\cdot})^t \Phi(z)  =\Ker[\mathbf{x}]_{\mathrm{H}_\mathbf{y},\gamma,1} (z), \quad \forall z\in \R^d
$$
\noindent for $(\gamma,\lambda)={\rm e}_{i_1}+{\rm e}_{n+1}$, $i_1=1, \ldots, n$, where $\mathrm{e}_j$ denotes the $j$-th canonical $(n+1)$-dimensional vector, for $j=1,\ldots, n$.
The above discussion shows that the standard Euclidean kernel is a particular case of our multidimensional kernel.
\end{rmk}

The following example illustrates the construction of the kernel operator for a given  transformation $\Phi$.

\begin{ex}\label{ex:2}
Let us consider six  points $[\mathbf{x}]=\Big((0,0), (0,1), (1,0)$, $(1,1)$, $(1,-1)$, $(-1,1)\Big)$ on the plane with patterns $\mathbf{y}=(1,1,1,-1,-1,-1)$. The points are drawn in Figure \ref{fig:1} where the 1-class points are identified with filled dots while the $-1$-class is identified with circles.  Clearly, the classes  are not linearly separable.

 Consider  the transformation  $\Phi: \R^2 \rightarrow \R^3$, defined as
$$
\Phi(x_1,x_2)=(x_1^2, \sqrt[r]{2}x_1x_2, x_2^2), \quad \forall (x_1, x_2) \in \R^2.
$$
\begin{figure}[h]
\begin{center}
\begin{tikzpicture}[scale=1.5]

\coordinate(X1) at (0,0);
\coordinate(X2) at (0,1);
\coordinate(X3) at (1,0);
\coordinate(X4) at (1,1);
\coordinate(X5) at (1,-1);
\coordinate(X6) at (-1,1);

\draw[latex-latex, thin, draw=gray] (-1.5,0)--(1.5,0) node [right] {$x_1$}; 
 \draw[latex-latex, thin, draw=gray] (0,-1.5)--(0,1.5) node [above] {$x_2$}; 

\foreach \Point in {(X1), (X2), (X3)}{
    \node at \Point {\textbullet};
}

\foreach \Point in {(X6),(X4),(X5)}{
    \node at \Point {$\circ$};
}

\end{tikzpicture}
\end{center}
\caption{Points of Example \ref{ex:2} and their classification patterns. \label{fig:1}}
\end{figure}
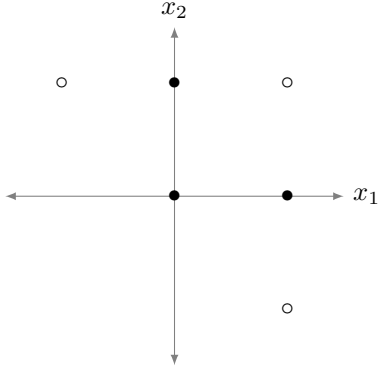

Mapping the six points  using $\Phi$, we get that, for any nonnegative integer $r$, the signs appearing at the kernel expressions are: $\sign(\alpha_3-\alpha_4-\alpha_5-\alpha_6)$, $\sign(-\sqrt[r]{2} \alpha_4 + \sqrt[r]{2} \alpha_5 + \sqrt[r]{2} \alpha_6)$ and $\sign( \alpha_2- \alpha_4 - \alpha_5 -\alpha_6)$. Since  $\mathrm{H}_\mathbf{y} = \{\alpha \in \R^6_+: \alpha_1+\alpha_2+\alpha_3-\alpha_4-\alpha_5-\alpha_6=0\}$, we get that the signs can be simplified to:
\begin{itemize}
\item $\sign(\alpha_3-\alpha_4-\alpha_5-\alpha_6) = \sign(-\alpha_1-\alpha_2)=-1$, if $\alpha_1+\alpha_2 >0$ and $1$ otherwise,
\item $\sign(-\sqrt[r]{2} \alpha_4 + \sqrt[r]{2} \alpha_5 + \sqrt[r]{2} \alpha_6) = \sign(\alpha_5+\alpha_6-\alpha_4)$.
\item $\sign(\alpha_2- \alpha_4 - \alpha_5 -\alpha_6) = \sign(-\alpha_1-\alpha_3)=-1$,   if $\alpha_1+\alpha_3 >0$ and $1$ otherwise.
\end{itemize}
 Observe that the cases where the argument within the $sign$ function is zero do not affect the formulations since the corresponding factor is null.
 Hence, only the expressions for the signs when $\alpha_1+\alpha_2 >0$ (in the first item) and $\alpha_1+\alpha_3>0$ (in the third item) are considered.

For odd $r$ (in which the $r$-th power of the signs above coincide with the signs themselves), we define the suitable subdivision $\{R_1, R_2\}$, where:
$$
R_1 = \{\alpha\in \mathrm{H}_\mathbf{y}: \alpha_5 + \alpha_6 \geq \alpha_4\} \mbox{ and } R_2 = \{\alpha\in \mathrm{H}_\mathbf{y}: \alpha_5 + \alpha_6 \leq \alpha_4\}.
$$
Note that $ {\rm S}_{\Phi}(R_1)=\{(-1,1,-1)\}$ while ${\rm S}_{\Phi}(R_2)= \{(-1,-1,-1)\}$,  i.e,  $\{R_1,R_2\}$ is a suitable $\Phi$-subdivision.

Thus,
$$
\Ker[\mathbf{x}]_{R_k,\gamma, \lambda} (z) = \left\{
\begin{array}{cl}
-\Phi_1(\mathrm{x})^{\gamma}\Phi_1(z)^\lambda+\Phi_2(\mathrm{x})^{\gamma}\Phi_2(z)^\lambda-\Phi_3(\mathrm{x})^{\gamma}\Phi_3(z)^\lambda, & \mbox{if $k=1$,}\\
-\Phi_1(\mathrm{x})^{\gamma}\Phi_1(z)^\lambda-\Phi_2(\mathrm{x})^{\gamma}\Phi_2(z)^\lambda-\Phi_3(\mathrm{x})^{\gamma}\Phi_3(z)^\lambda, & \mbox{if $k=2$},\end{array}\right.
$$
being then:
$$
\Ker[\mathbf{x}]_{R_k,\gamma, \lambda} (z) = \left\{
\begin{array}{cl}
-\left(\mathrm{x}_{\cdot 1}^\gamma z_{1}^\lambda - \mathrm{x}_{\cdot 2}^\gamma z_{2}^\lambda\right)^2, & \mbox{if $k=1$,}\\
-\left(\mathrm{x}_{\cdot 1}^\gamma z_{1}^\lambda + \mathrm{x}_{\cdot 2}^\gamma z_{2}^\lambda\right)^2, & \mbox{if $k=2$}.\end{array}\right.
$$
For even $r$, because the $r$-th power of the signs do not affect to the expressions, the $r$-order Kernel function of $\Phi$ is given by
$$
\Ker[\mathbf{x}]_{R_k,\gamma, \lambda} (z) =\left(\mathrm{x}_{\cdot 1}^\gamma z_{1}^\lambda + \mathrm{x}_{\cdot 2}^\gamma z_{2}^\lambda\right)^2,
$$
for $k=1, 2$, $(\gamma,\lambda) \in \N_r^{n+1}$ and $\lambda \in \{0,1\}$.
\hfill $\square$
\end{ex}

\subsection{Multidimensional Kernels and higher-dimensional tensors}

Given a subdivision $\{R_k\}_{k \in \mathcal{K}}$  of $\mathrm{H}_\mathbf{y}$ and a set of functions $\{\Ker[\mathbf{x}]_{R_k,\gamma,\lambda}\}_{k\in \mathcal{K}}$, for any $(\gamma,\lambda) \in \N_r^{n+1}$ with $\lambda \in \{0,1\}$,
 the \textit{critical} question in this section is the existence of $D \in \Z_+$ and
 $\Phi: \mathbb{R}^d \longrightarrow \mathbb{R}^D$ such that, $\{R_k\}_{k \in \mathcal{K}}$ is a suitable $\Phi$-subdivision and
\begin{eqnarray*}
\Ker[\mathbf{x}]_{R_k,\gamma,\lambda} (z):= \sum_{j=1}^D
\mathrm{s}_{R_k,j}^r \Phi_j(\mathrm{x})^{\gamma} \Phi_j(z)^{\lambda}, \quad \forall z \in \R^d,
\end{eqnarray*}
where $ {\rm S}_{\Phi}(R_k)=\{s_{R_k}\}$ with $s_{R_k} \in \{-1,1\}^D$.

 For the sake of simplicity in the formulations, each of the elements of the $\Phi$-suitable subdivision of  $\mathrm{H}_\mathbf{y}$ will be denoted as follows:
$$
R_k = \{\alpha \in \R^n: M^k_j \alpha \geq 0, j=1, \ldots, m_k\},
$$
 where $M^k_j \in \R^{n}$, for $k\in \mathcal{K}$ and $j=1,\dots,m_k$.

First of all, using that $\Ker[\mathbf{x}]_{R_k,\gamma,\lambda} (z)$ is  a $r$-order Kernel
function of $\Phi$, the problem \eqref{fo}-\eqref{fin} for a transformation of the original
data via $\Phi$, in each element,  $k \in \mathcal{K}$,  of a suitable $\Phi$-subdivision can be written  as:
\begin{align}
\label{fo11}
\max \; & F_k(\alpha):=
\left(\dfrac{1}{p^{r}} - \dfrac{1}{p^{r-1}}\right) \dsum_{\gamma \in \N^n_r} c_{\gamma} \alpha^{\gamma} y^{\gamma}   \Ker[\mathbf{x}]_{R_k,\gamma,0} (z)+ \sum_{i=1}^n \alpha_i\\
\mbox{s.t. } \: & \tilde{g}^k_{j}(\a) :=
      M^k_j \alpha \geq 0, \qquad \forall j=1, \ldots, m_k,\\
       & \tilde{\ell}_0(\a):=     \dsum_{i=1}^n \alpha_i y_i =0, \\
       & \tilde{\ell}_i(\a):=
       C - \alpha_i\ge 0,   \qquad \forall i=1,\ldots,n,\\
\label{fin11}
       &   \tilde{\ell}_{n+i} (\a):=\a_i\ge 0, \qquad \quad \forall i=1,\ldots,n.
\end{align}
Observe that the problem above  is a reformulation of the Lagrangian dual problem, \eqref{of:lagrangean}, for the  $\Phi$-transformed data that only depends on the original data via the $r$-order Kernel function of $\Phi$ and the suitable $\Phi$-subdivision,  and it can be seen as an extension of the kernel trick to $\ell_p$-norms with $p > 1$.

In the particular case where $\Phi$ is the identity transformation,  the above formulation  becomes \eqref{fo}-\eqref{fin} whenever  the suitable $\Phi$-subdivision consists of  the full dimensional elements of the  arrangement of hyperplanes $\{\sum_{i=1}^n \alpha_i y_i x_{ij}=0, \,j=1,\ldots,d\}$.
Furthermore, observe that  $F_k$ does not depend on $z$, since the degree, $\lambda$, of such a value is zero in that function.

We shall connect the above mentioned \textit{critical} question with some interesting  mathematical objects, \textit{real symmetric tensors}, that are built upon the given data set $[\mathbf{x}]$ and $\mathbf{y}$. It will become clear, after Theorem \ref{conditionskernel}, that existence of a kernel operator is closely related with rank-one decompositions of the above mentioned tensors.

Recall that a real $r$-th order $m$-dimensional symmetric tensor, $\mathbb{L}$, consists of $m^r$ real entries $\mathbb{L}_{j_1\ldots j_r} \in \R$ such that $\mathbb{L}_{j_1\ldots j_r} = \mathbb{L}_{j_{\sigma(1)} \ldots {j_{\sigma(1)}}}$, for any permutation  $\sigma$ of $\{1, \ldots, r\}$.

\begin{lem}
Let $\Phi: \mathbb{R}^d \longrightarrow \mathbb{R}^D$, $\hat z\in \R$ and let $\mathcal{S}=\{R_k\}_{k \in \mathcal{K}}$ be a suitable $\Phi$-subdivision of $\mathrm{H}_\mathbf{y}$. Then, the $k$-th slice of any $r$-order kernel function of $\Phi$ at $\hat z$, induces a real $r$-th order $(n+1)$-dimensional symmetric tensor.
\end{lem}
\begin{proof}
Let us define the following set of $(n+1)^r$ real numbers:
$$
\K^k_{i_1\ldots i_r} = \left\{\begin{array}{cl}
\Ker[\mathbf{x}]_{R_k,\gamma_0,0} (\hat{z}), &\mbox{if $i_j < n+1$, $\forall j=1,\ldots,r$,}\\
\Ker[\mathbf{x}]_{R_k,\gamma_1,1} (\hat{z}), & \mbox{if there exists $s\in \{1, \ldots, r\}$ such that $i_s=n+1$.}
	\end{array}\right.
$$
being $(\gamma_0,\lambda)=\sum_{l=1}^r  {\rm e}_{i_l}$ with $\lambda=0$ and $(\gamma_1,\lambda)=\sum_{l=1}^r  {\rm e}_{i_l}$ with $\lambda=1$ .

Let us check whether the above tensor is symmetric. Let $\sigma$ be a permutation of the indices. For $(i_1, \ldots, i_r)$, which comes from a particular choice of $(\gamma, \lambda)$, if $\sigma$ is applied to $(i_1, \ldots, i_r)$, the resulting $(\gamma^\prime,\lambda^\prime)$ becomes:
$$
(\gamma^\prime,\lambda^\prime) = \hspace*{-0.1cm} \left\{\begin{array}{cl}
\dsum_{l=1}^r  {\rm e}_{{\sigma(l)}}, &\mbox{if $i_{\sigma(i)} < n+1, \forall i$,}\\
\dsum_{l=1}^r  {\rm e}_{{\sigma(l)}},  & \mbox{if $\exists s: i_{\sigma(s)}=n+1$}
\end{array}\right. \hspace*{-0.2cm}=\left\{\begin{array}{cl}
\dsum_{l=1}^r  {\rm e}_{i_l}, &\mbox{if $i_{i} < n+1, \forall i$,}\\

\dsum_{l=1}^r  {\rm e}_{i_l},  & \mbox{if $\exists s:i_s=n+1$}
\end{array}\right. = (\gamma,\lambda)
$$
Hence, $\mathbb{K}_{i_1\ldots i_r} = \mathbb{K}_{i_{\sigma(1)} \ldots i_{\sigma(r)}}$, since the multi-indices constructed from $(\gamma,\lambda)$ and $(\gamma^\prime,\lambda^\prime)$ coincide.
\end{proof}

Let us now denote by $\otimes$ the tensor product, i.e. $v \otimes w = (v_i\, w_j)_{i,j=1}^m$ for any $v, w \in \R^m$.

\begin{lem}[\cite{comon:hal-00327599}]\label{lemma:common}
Let $\mathbb{K}$ be a real $r$-order $(n+1)$-dimensional symmetric tensor. Then, there exists $\widehat{D} \in \N$,  $v_1, \ldots, v_{h} \in \R^{n+1}$ and $\psi_j \in \R$ $\forall j=1,\ldots,{\widehat{D}}$,
such that $\mathbb{K}$ can be decomposed as
\begin{equation*}\label{decomp}
\mathbb{K} = \dsum_{j=1}^{\widehat{D}} \psi_j \; v_j \otimes \stackrel{r}{\cdots} \otimes v_j \,  .
\end{equation*}
 That is, $\K_{i_1\cdots i_r} = \dsum_{j=1}^{\widehat{D}} \psi_jv_{j i_1} \cdots v_{j i_r}$ for any $i_1, \ldots, i_r \in \{1, \ldots,n+1\}$. Such a decomposition is said a rank-one tensor decomposition of $\mathbb{K}$. The minimum $\widehat{D}$ that assures such a decomposition is the symmetric tensor rank and $\psi_1, \ldots,  \psi_{\widehat{D}}$ are its eigenvalues.
\end{lem}

The following result extends the classical Mercer's Theorem \cite{mercer} to $r$-order Kernel functions.
\begin{thm}\label{conditionskernel}
Let $\{R_k\}_{k\in\mathcal{K}}$ be a subdivision of $\mathrm{H}_\mathbf{y}$
and $\mathbb{K}^k$, for $k\in \mathcal{K}$, be a $r$-order $(n+1)$-dimensional symmetric tensor such that each $\mathbb{K}^k$ can be decomposed as:
$$
\mathbb{K}^k = \dsum_{j=1}^{\widehat{D}} \psi_{kj} v_j \otimes \stackrel{r}{\cdots} \otimes v_j, \: \forall k \in \mathcal{K},
$$
and satisfying, either
\begin{enumerate}
\item $r$ is even and $\psi_{j}:= \psi_{kj} \geq 0$, or
\item $r$ is odd and $\psi_j:= |\psi_{kj}|$ and for all $k\in \mathcal{K}$:
$$
\sign(\psi_{kj}) = \sign\Big( \sum_{i=1}^n  \alpha_i y_i \sqrt[r]{\psi_j} v_{ji} \Big), \text{ for all } \alpha \in \R_k.
$$
\end{enumerate}
Then, there exists a transformation $\Phi$, such that $\{R_k\}_{k\in \mathcal{K}}$ is a $\Phi$-suitable subdivision of $\mathrm{H}_\mathrm{y}$ and $\{\mathbb{K}^k\}_{k\in \mathcal{K}}$ induces a $r$-order kernel function of $\Phi$.
\end{thm}
\begin{proof}
Let $z \in \R^d$ and define $\Phi: \R^d \rightarrow  \R^{\widehat{D}}$ as:
$$
\left\{\begin{array}{ccl}
\Phi_j(\mathrm{x}_{i\cdot}) &=&  \sqrt[r]{\psi_j} v_{ji},  \mbox{ for } i=1, \ldots, n,\\
\Phi_j(\;z\;) &=& \sqrt[r]{\psi_j} v_{j,n+1},
\end{array}\right. \mbox{ for $j=1, \ldots, \widehat{D},$}
$$
which is well defined because of the nonnegativity of the eigenvalues $\psi_j$.

\begin{itemize}
\item
Let us assume first that $r$ is even.
 Note that, since $r$ is even and $\{R_k\}_{k\in \mathcal{K}}$ is a suitable subdivision, the latest is also a  suitable $\Phi$-subdivision (actually, for any $\Phi$), since the signs are always positive (or the sign function is always 1).

Hence, for $(\gamma,\lambda)= \sum_{l=1}^r{\rm e}_{i_l}$,
$$
\Ker[\mathbf{x}]_{R_k,\gamma, \lambda} (z)   = \sum_{j=1}^{\widehat{D}}
\Phi_j(\mathrm{x})^{\gamma} \Phi_j(z)^{\lambda} = \dsum_{j=1}^{\widehat{D}}  \psi_j v_{j i_1} \cdots v_{j i_r} = \mathbb{K}^k_{i_1\ldots i_r}, 
$$

is a $r$-order kernel function of $\Phi$.
\item Assume that  $r$ is odd.
Observe that $\sign\Big( \sum_{i=1}^n  \alpha_i y_i \Phi_j(\mathrm{x}_{i\cdot})\Big)$\\ $= \sign\Big( \sum_{i=1}^n  \alpha_i y_i \sqrt[r]{\psi_j} v_{ji} \Big) = \sign(\psi_{kj})$, being then 
$$
{\rm S}_{\Phi}(R_k)=\{(\sign(\psi_{k1}),\ldots, \sign(\psi_{k\widehat{D}}))\}.
$$
Thus, we get that  $\{R_k\}_{k\in \mathcal{K}}$ is a  suitable $\Phi$-subdivision of $\mathrm{H}_\mathbf{y}$.

Also, because $\psi_{kj} = \sign(\psi_{kj}) \psi_j$ and
$\sign(\psi_{kj}) = \sign\Big( \sum_{i=1}^n  \alpha_i y_i \sqrt[r]{\psi_j} v_{ji} \Big)$, for all $\alpha \in \R_k$, we get that:
\begin{eqnarray*}
\Ker[\mathbf{x}]_{R_k,\gamma, \lambda} (z) &=&  \sum_{j=1}^{\widehat{D}} \sign\Big( \sum_{i=1}^n \alpha_i y_i  \Phi_j(\mathrm{x}_{i\cdot}) \Big)^r \Phi_j(\mathrm{x})^{\gamma} \Phi_j(z)^{\lambda}\\
&=& \sum_{j=1}^{\widehat{D}} \sign\Big( \sum_{i=1}^n \alpha_i y_i   \sqrt[r]{\psi_j} v_{ji} \Big)^r  \psi_{j} v_{j i_1} \cdots v_{j i_r}\\
&=&\mathbb{K}^k_{i_1\ldots i_r}
\end{eqnarray*}

for $(\gamma,\lambda)=\sum_{l=1}^r {\rm e}_{i_l}$. Hence, $\mathbb{K}^k$ induces a $k$th-slice of a $r$-order kernel function of $\Phi$.
\end{itemize}

\end{proof}

The decomposition of symmetric $2$-order $n$-dimensional tensors ($n\times n$ symmetric matrices) provided in Lemma \ref{lemma:common}, is equivalent to eigenvalue decomposition \cite{eckartand-young} and the symmetric tensor rank coincides with the usual rank of a matrix. Hence, for $r=2$ (Euclidean case), the conditions of Theorem \ref{conditionskernel} reduce to check positive semidefiniteness of the induced kernel matrix (Mercer's Theorem). On the other hand, computing rank-one decompositions of higher-dimensional symmetric tensors is known to be NP-hard, even for symmetric $3$-order tensors \cite{hillar-lim2013}. Actually, there is no finite algorithm to compute, in general, the rank one decompositions of general symmetric tensors. In spite of that, several algorithms have been proposed to perform such a decomposition. One commonly used strategy finds approximations to the decomposition by sequentially increasing the dimension of the transformed space ($\widehat{D}$). Specifically, one fixes a dimension $\widehat{D}$ and finds $v$ and $\psi$ that minimize $\|\K - \sum_{j=1}^{\widehat{D}} \psi_{kj} v_j \otimes \stackrel{r}{\cdots} \otimes v_j\|_2^2$. Next, if a zero-objective value is obtained, a tensor decomposition is found; otherwise, $\widehat{D}$ is increased and the process is repeated. The interested reader referred to \cite{als,asd,kofidis} for further information about algorithms for  decomposing real
symmetric tensors.

 In some  interesting cases, the assumptions of Theorem \ref{conditionskernel} are proved to be verified by some general classes of tensors. In particular, even order $P$ tensors, $B$ tensors, $B_0$ tensors, diagonally dominated tensors, positive Cauchy tensors and sums-of-squares (SOS) tensors are known to have all their eigenvalues nonnegative (the reader is referred to \cite{SOStensors,Btensors} for the definitions and results on
 these families of tensors). Thus, several classes of  multidimensional kernel functions can be easily constructed.  For instance, if $r$ is even and we assume that all $\mathrm{x}_{i\cdot}\neq \mathbf{0}$, for all $i=1,\ldots, n+1$, it is well-known that the symmetric $r$-order $(n+1)$-dimensional tensor, $\mathbb{K}$, with entries:
$$
\mathbb{K}_{i_1\ldots i_r} = \dfrac{1}{\|\mathrm{x}_{i_1\cdot}\|+\cdots+\|\mathrm{x}_{i_r\cdot}\|}, \qquad i_1, \ldots, i_r=1, \ldots, n+1,
$$
for some norm $\|\cdot\|$ in $\R^d$, is a Cauchy-shaped tensor. Next, $\mathbb{K}$ is  positive semidefinite~\cite{Cauchytensor}, since $\|\mathrm{x}_{i\cdot}\| > 0$, for all $i=1,\ldots, n$. Hence, by Theorem  \ref{conditionskernel}, it induces a $r$-order kernel function.

\section{Solving the $\ell_p$-SVM problem}\label{sec:4}

By Theorem \ref{t:12}, \eqref{of:lagrangean} can be rewritten as  a polynomial optimization problem, i.e., as a global optimization problem and then, in general, it is NP-hard. In spite of that we can approximately solve moderate size instances using a modern optimization technique based on the Theory of Moments that allows building a sequence of semidefinite  programs whose solutions converge, in the limit, to the optimal solution of the original problem \cite{lasserrebook}. We use it to derive upper bounds for the Lagrangian dual problem \eqref{of:lagrangean}.

Let us denote by $\R[\a]$ the ring of real polynomials in the variables $\a=(\a_1$, $\ldots,\a_n)$,  for $n \in \N$ ($n \geq 1$), and by $\R[\a]_r \subset \R[\a]$ the space of polynomials of degree at most $r \in \N$. We also denote by $\mathcal{B} = \{\a^\gamma: \gamma\in\N^n\}$ a canonical basis of monomials for $\R[\a]$, where $\a^\gamma = \a_1^{\gamma_1} \cdots \a_n^{\gamma_n}$, for any $\gamma \in \N^n$.  Note that $\mathcal{B}_r = \{\a^\gamma \in \mathcal{B}: \sum_{i=1}^n \gamma_i \leq r\}$ is a basis for $\R[\a]_r$.

For any sequence indexed in the canonical monomial basis $\mathcal{B}$, $\mathbf{w}=(\w_\gamma)_{\gamma \in \N^n}\subset\R$, let $\L_\mathbf{w}:\R[\a]\rightarrow\R$ be the linear functional defined, for any $f=\sum_{\gamma\in\N^n}f_\gamma\,\a^\gamma \in \R[\a]$, as $\L_\mathbf{w}(f) :=
\sum_{\gamma\in\N^n}f_\gamma\,\w_\gamma$.

The \textit{moment} matrix $\Mo_r(\mathbf{w})$ of order $r$ associated with $\mathbf{w}$, has its rows and columns indexed by  the elements in the basis  $\mathcal{B}_r$
and for two elements in such a basis, $b_1=\a^\gamma, b_2=\a^\beta$, $\Mo_r(\mathbf{w})(b_1,b_2) = \Mo_r(\mathbf{w})(\gamma,\beta)\,:=\,\L_\mathbf{w}(\a^{\gamma+\beta})\,=\,\omega_{\gamma+\beta}$, for $\vert\gamma\vert,\,\vert\beta\vert\,\leq r$ (here $|a|$ stands for the sum of the coordinates of $a \in \N^n$). Note that the moment matrix of order $r$ has dimension $\bfrac{n+r}{n}\times\bfrac{n+r}{n}$ and that the number of $\w_\gamma$ variables is $\bfrac{n+2r}{n}$.

For $g = \sum_{\zeta \in \N^n}  g_{\zeta} \a^\zeta \in \R[\a]$, the \textit{localizing} matrix $\Mo_r(g \mathbf{w})$ of order $r$ associated with $\mathbf{w}$ and $g$, has its rows and columns indexed by the elements in $\mathcal{B}$ and for $b_1=\a^\gamma$, $b_2=\a^\beta$,
$\Mo_r(g\mathbf{w})(b_1, b_2) = \Mo_r(g\mathbf{w})(\gamma,\beta):=\L_\mathbf{w}(\a^{\gamma+\beta}g(\a))=\sum_{\zeta}g_\zeta \w_{\zeta+\gamma+\beta}$, for
 $\vert\gamma\vert,\vert\beta\vert\,\leq r$.
Observe that a different choice for the basis of $\R[\a]$, instead of the standard monomial basis, would give different moment and localizing matrices, although the results would be also valid.

The main assumption to be imposed when one wants to assure convergence of some SDP relaxations for solving polynomial optimization problems  is known as the Arquimedean property (see for instance \cite{lasserrebook}) and it is a consequence  of Putinar's results~\cite{putinar}.
The importance of Archimedean property
stems from the link between such a condition with the  positive semidefiniteness of the moment and localizing matrices  (see \cite{putinar}).

We built a hierarchy of SDP relaxations \textit{`a la Lasserre'}  for solving the dual problem. We observe that some constraints in this problem are already semidefinite (linear) therefore it is not mandatory to create their associated localizing constraints although its inclusion reinforces the relaxation values.

In the following result we state the semidefinite programming relaxations of the Lagrangean dual problem \eqref{fo11}-\eqref{fin11}. Obviously, this result can be easily  applied to problem \eqref{fo}-\eqref{fin} when we consider that $\Phi$ is the identity transformation.

\begin{thm}\label{thm:sdp}
Let $r \in \Z_+$ and $\{R_k\}_{k\in \mathcal{K}}$ be a suitable $\Phi$-subdivision of $\mathrm{H}_\mathbf{y}$. Let $t\ge t_0=\lceil \frac{r}{2}\rceil$, and
\begin{align*}
\rho_t^k= \inf _{\mathbf{w}}\; & \L_\mathbf{\mathbf{w}}(-F_k)\\
\quad \mbox{s.t. } & \Mo_t(\mathbf{w})\succeq 0,\\
 &\Mo_{t-1}(\tilde{g}^k_{j} \mathbf{w})\succeq 0, \;  j=0,\ldots, m_k, \\
 &\Mo_{t-1}(\tilde{\ell}_i \mathbf{w})\succeq 0, \; i=0,\ldots, n, \\
  &\Mo_{t-1}(\tilde{\ell}_{n+i} \mathbf{w})\succeq 0, \; i=0,\ldots, n, \\
 &\L_\mathbf{w}(\w_\mathbf{0})=1.
\end{align*}
Then, the sequence $\{\rho_t^k\}_{t\ge t_0}$ of optimal values of the hierarchy of problems above satisfies
$$ \lim_{t\rightarrow +\infty} -\rho_t^k \downarrow \max_{\a \in\R^n_+}F_k(\a).$$
\end{thm}
\begin{proof}
We apply Lasserre's hierarchy of  SDP relaxations to approximate the dual global optimization problem. The feasible domain is compact since $\a$ belongs to a closed and bounded set (recall that, in particular, $\a \in \mathrm{H}_\mathrm{y} \cap R_k$, so $\a \in [0,C]^n$) and therefore it satisfies the Archimedean property \cite{putinar}. Next, the maximum degree of the polynomials involved in the problem is $r$. Therefore we can apply \cite[Theorem 5.6]{lasserrebook} with relaxation orders $t\ge t_0:=\lceil r/2\rceil$ to conclude that the sequence, $\{-\rho_t^k\}_{t\ge t_0}$, of optimal values of the SDP relaxations, in the statement of the theorem, converges to $\dmax_{\a\in\R_+^n} F_k(\a)$, the optimal value of the Lagrangian problem.
\end{proof}
 In many cases the convergence ensured by the above theorem is attained in a finite number of steps and it can be certified by a sufficient condition called the \textit{rank condition} \cite[Theorem 6.1]{lasserrebook}, implying that the optimal $\alpha$-optimal values can be extracted.

\begin{ex}\label{ex:3}
Let us illustrate the proposed moment-SDP methodology to the toy instance of Example \ref{ex:2}. Recall that  we apply, for odd $r$, the kernel function $\Ker[\mathbf{x}]_{R_k,\gamma, \lambda} (z)$ given by:
$$
\Ker[\mathbf{x}]_{R_k,\gamma, \lambda} (z) = \left\{
\begin{array}{cl}
-\left(\mathrm{x}_{\cdot 1}^\gamma z_{1}^\lambda - \mathrm{x}_{\cdot 2}^\gamma z_{2}^\lambda\right)^2, & \mbox{if $k=1$,}\\
-\left(\mathrm{x}_{\cdot 1}^\gamma z_{1}^\lambda + \mathrm{x}_{\cdot 2}^\gamma z_{2}^\lambda\right)^2, & \mbox{if $k=2$}.\end{array}\right.
$$
where $R_1 = \{\alpha\in \mathrm{H}_\mathbf{y}: \alpha_5 + \alpha_6 \geq \alpha_4\}$ and  $R_2 = \{\alpha\in \mathrm{H}_\mathbf{y}: \alpha_5 + \alpha_6 \leq \alpha_4\}$.

For $r=3, s=1$, i.e., when $p=\frac{3}{2}$ and $q=3$, the following two problems have to be solved:
\begin{align*}
\max &\left(\dfrac{1}{\left(\frac{3}{2}\right)^{3}} - \dfrac{1}{\left(\frac{3}{2}\right)^{2}}\right)  \dsum_{\gamma \in \N^6_3} c_{\gamma} \alpha^{\gamma} y^{\gamma}   \Ker[\mathbf{x}]_{R_k,\gamma,0} (z)+ \sum_{i=1}^6 \alpha_i\\
\mbox{\rm s.t. } &  \alpha \in \mathrm{H}_{\mathrm{y}} \cap R_k.
\end{align*}
for $k=1, 2$.

For the $1$-th slide of the suitable subdivision, $R_1$, the problem is explicitly expressed as:
\begin{align*}
\max \; & F_1(\alpha) = {\small\frac{-4}{27}}  \left( -\alpha_2^3+3\alpha_2^2\alpha_4+3\alpha_2^2\alpha_5+3\alpha_2^2\alpha_6-3\alpha_2\alpha_4^2-6\alpha_2\alpha_4\alpha_5-6\alpha_2\alpha_4\alpha_6-\right.\\
&3\alpha_2\alpha_5^2-6\alpha_2\alpha_5\alpha_6-3\alpha_2\alpha_6^2-\alpha_3^3+3\alpha_3^2\alpha_4+3\alpha_3^2\alpha_5+3\alpha_3^2\alpha_6-3\alpha_3\alpha_4^2-6\alpha_3\alpha_4\alpha_5-\\
&6\alpha_3\alpha_4\alpha_6-3\alpha_3\alpha_5^2-6\alpha_3\alpha_5\alpha_6-3\alpha_3\alpha_6^2+12\alpha_4^2\alpha_5+12\alpha_4^2\alpha_6+4\alpha_5^3+12\alpha_5^2\alpha_6+\\
&\left.12\alpha_5\alpha_6^2+4\alpha_6^3\right)+( \alpha_1+\alpha_2+\alpha_3+\alpha_4+\alpha_5+\alpha_6)\\
\mbox{\rm s.t. } & \alpha_1+\alpha_2+\alpha_3-\alpha_4-\alpha_5-\alpha_6=0,\\
& \alpha_5 + \alpha_6 \geq \alpha_4,\\
& 0 \leq \alpha_1,\alpha_2, \alpha_3, \alpha_4, \alpha_5, \alpha_6 \leq 10.
\end{align*}
We use \texttt{Gloptipoly} 3.8 \cite{gloptipoly} to translate the above problem into the SDP-relaxed problem and \texttt{SDPT3}\cite{sdpt3} as the semidefinite programming solver.

Note that since the degree of the multivariate polynomial involved in the objective function is $3$, at least a relaxation order of $2$ is needed for the moment matrices. Using the basis $\mathcal{B}_2$, the moment matrix of order $2$ has the following shape:

$\Mo_{2}(\mathbf{w})=
  \begin{blockarray}{*{7}{c} l}
    \begin{block}{*{7}{>{$\footnotesize}c<{$}} l}
      1& $\alpha_1$ &  $\cdots$ & $\alpha_6$ & $\alpha_1^2$ & $\cdots$ & $\alpha_6^2$\\
    \end{block}
    \begin{block}{[*{7}{c}]>{$\footnotesize}l<{$}}
w_{000000} & w_{100000} & \cdots & w_{000001} & w_{200000} & \cdots & w_{000002} & 1\\
w_{100000} & w_{200000}&  \cdots & w_{100001} & w_{300000} & \cdots & w_{100002} & $\alpha_1$\\
 & & \ddots &&&& &$ \vdots$ \\
w_{000002} &&&&&&  w_{000004} & $\alpha_6^2$\\
    \end{block}
  \end{blockarray},$

that is a $28\times 28$ real matrix and in which $210$ variables are involved.

Observe that the constraints of the problems are linear, so no need of relaxing the constrains or using localizing matrices is needed. The semidefinite problem to solve is:
\begin{align*}
\rho^1_2 = \min \L_\mathbf{\mathbf{w}}(-F_1)\\
\mbox{\rm s.t. } & \Mo_{2}(\mathbf{w}) \succ 0,\\
& w_{100000}+w_{010000}+w_{001000}-w_{000100}-w_{000010}-w_{000001}=0,\\
& w_{000010} + w_{000001} \geq w_{000100},\\
& 0 \leq w_{100000}, w_{010000}, w_{001000},  w_{000100}, w_{000010}, w_{000001} \leq 10,
\end{align*}
where in $ \L_\mathbf{\mathbf{w}}(-F_1)$ each term $\alpha^\gamma$ is transformed into $w_\gamma$, for $\gamma \in \N^n_r$.

Solving the above problem, we get $\rho_2^1=-5.6569$ and:
$$
w_{100000} =0, w_{010000} =w_{001000}  = w_{000100}=2.1213, w_{000010}=w_{000001}=1.0611.
$$
Also, \texttt{Gloptipoly} checked that the rank condition holds, certifying that $\alpha^*=(0,$ $2.1213$, $2.1213,$ $2.1213,1.0611,1.0611)$ is optimal for our problem.

The problem for the second subdivision, $R_2$, can be analogously stated (by considering $\alpha_5+\alpha_6 \leq \alpha_4$ instead of the one defining $R_1$). Also, for relaxation order $2$, \texttt{Gloptipoly} obtained an optimal value of $\rho_2^2= -5.6569$ (the same objective value as for $R_1$), and the solution was certified to be optimal with the same values as in the problem for $k=1$ (observe that the obtained optimal solution belongs to both $R_1$ and $R_2$ since $\alpha_5^*+\alpha_6^*=\alpha_4^*$).

The optimal separating hyperplane is now constructed from $\alpha^*$. First, the
 intercept, $b$, is derived by using an observation $i_0$ such that $0<\alpha_{i_0}<10$, for instance taking $i_0=2$ ($y_2=1$ and $\mathrm{x}_{2\cdot}=(0,1)$), being

$$
b= y_2 - \frac{1}{\left(\frac{3}{2}\right)^{2}} \dsum_{\gamma \in \N^6_2} c_{\gamma} (\alpha^*)^{\gamma} y^{\gamma}   \Ker[\mathbf{x}]_{R_1,\gamma,1} (\mathrm{x}_{2\cdot}) = 1 - \frac{1}{\left(\frac{3}{2}\right)^{2}} \dsum_{\gamma \in \N^6_2} c_{\gamma} (\alpha^*)^{\gamma} y^{\gamma}   \mathrm{x}_{\cdot 1}^{2\gamma} = 3
$$
Thus, the hyperplane has the following shape:
\begin{eqnarray*}
\mathcal{H} &= \left\{z \in \R^2: \frac{1}{\left(\frac{3}{2}\right)^{2}} \dsum_{\gamma \in \N_{2}^{6}} c_{\gamma} \alpha^{* \gamma} y^{\gamma} \Ker[\mathbf{x}]_{R_1,\gamma,1} (z)+3 = 0\right\}\\
&= \left\{z \in\R^2: \frac{1}{\left(\frac{3}{2}\right)^{2}} \dsum_{\gamma \in \N_{2}^{6}} c_{\gamma} \alpha^{* \gamma} y^{\gamma} (\mathrm{x}_{\cdot 1}^\gamma z_1 - \mathrm{x}_{\cdot 2}^\gamma z_2)^2=-3\right\},
\end{eqnarray*}
and the data can be classified according to the sign of  the evaluation of each point  on the above hyperplane. Hence, for the six obtained points, we get that $\mathrm{x}_{1\cdot}, \mathrm{x}_{2\cdot}$, and $\mathrm{x}_{3\cdot}$ are classified in the $+1$ side of the separating hyperplane, while $\mathrm{x}_{4\cdot}, \mathrm{x}_{5\cdot}$, and $\mathrm{x}_{6\cdot}$ are classified in
the  $-1$ side. Thus, all the points are well-classified.

\end{ex}

\begin{rmk}
Theorem \ref{thm:sdp} can be extended to solve the Lagrangian dual problem related to the general case in which $q=\frac{r}{s}$, for $r, s \in \Z_+$ with $\gcd(r,s)=1$ and $q>1$ (see Remark \ref{rmk1}). In such a case, if $\{R_k\}_{k\in \mathcal{K}}$ is a suitable $\Phi$-subdivision of $\mathrm{H}_\mathbf{y}$ and $t\ge t_0=\lceil \frac{r}{2}\rceil$,
 we have to consider the following hierarchy of semidefinite programming problems
\begin{align*}
\tilde{\rho}_t^k= \inf _{\mathbf{w}}\; & \L_\mathbf{\mathbf{w}}(-\tilde{f}_{\mathrm{s}_{R_k}})\\
\quad \mbox{s.t. } & \Mo_t(\mathbf{w})\succeq 0,\\
 &\Mo_{t-1}(  \tilde{g}^k_{j} \mathbf{w})\succeq 0, \; j=1,\ldots,m_k \\
  &\Mo_{t-\lceil \frac{r}{2}\rceil}(-\tilde{g}^k_{m_k+j} \mathbf{w})\succeq 0, \; j=1,\ldots,D, \\
  &\Mo_{t-1}(\tilde{\ell}_i \mathbf{w})\succeq 0, \; i=0,\ldots, n, \\
    &\Mo_{t-1}(\tilde{\ell}_{n+i} \mathbf{w})\succeq 0, \; i=0,\ldots, n, \\
 &\L_\mathbf{w}(\w_\mathbf{0})=1,
\end{align*}
where, $\tilde{g}_{m_k+j}^k(\alpha,\delta)=
\dsum_{\gamma \in \N_r^n}  c_{\gamma} \alpha^{\gamma} y^{\gamma} \mathrm{s}_{R_{k,j}}^r  \mathrm{x}_{\cdot j}^{\gamma}-\delta_j^{s}$ for $j=1,\ldots,D$.
Then, it is satisfied that $$\dlim_{t\to +\infty} -\tilde{\rho}_t^k \downarrow \max_{\a,\delta} \tilde{f}_{\mathrm{s}_{R_k}}(\a,\delta).$$
Observe that the polynomial optimization problems in both cases, $q$ integer and rational, have a similar shape (and also their semidefinite programming relaxations).  However, for fractional $q$, there are $D$ variables $\delta_1, \ldots, \delta_D$, a well as $D$ constraints, $\tilde{g}^k_{m_k+1} (\alpha), \ldots, \tilde{g}^k_{m_k+D} (\alpha)$, defining   the feasible region (recall that $\tilde{g}^k_{1} (\alpha), \ldots, \tilde{g}^k_{m_k} (\alpha)$  describe the
cells of a suitable $\Phi$-subdivision), so the \textit{kernel trick} cannot be applied under this setting.
\end{rmk}

\subsection{Solving the primal $\ell_p$-SVM problem}
The above machinery  allows us to solve the dual  of the SVM problem, resorting to hierarchies of SDP problems. The main drawback of that approach is the increasing size of the SDP objects that have to be handled as the relaxation order of the problem grows. The current development of SDP solvers limits the applicability of that approach to problems with several hundreds of variables which may be an issue to handle big databases.

One way to overcome that inconvenience is to attack directly the primal problem. Our strategy in order to solve the primal problem will be the following. Let $\mathcal{C}_{\R^D}(T)$ be the Banach space of continuous functions from a compact set $T \subseteq\R^d $ to $\mathbb{R}^D$. It is well-known that $\mathcal{C}_{\R^D}(T)$ admits a Schauder basis (see \cite{LiTza77}). In particular, $\mathcal{B}=\{\mathrm{z}^\gamma: \gamma \in \N^d\}$, the standard basis of multidimensional monomials is a Schauder basis for this space. Also   Bernstein and trigonometric polynomials and some others are Schauder bases of this space. This means that any continuous function defined on $T$ can be exactly represented as a sum of terms in the basis $\mathcal{B}$ (sometimes infinitely many). Thus for any
 continuous  function $ \Phi:  T \longmapsto \mathbb{R}^D$, there exists an expansion such that $\Phi(\mathrm{z})=\sum_{j=1}^\infty \tau_j \mathrm{z}_j$, with $\tau_j\in \mathbb{R}$ and $\mathrm{z}_j\in \mathcal{B}$ for any $j=1, \ldots, \infty$.

These expansions are function dependent but one may expect that with a sufficient number of terms we can approximate up to a certain degree of accuracy the standard kernel transformations usually applied in SVM. In this regard, our solution strategy transforms the original data by using a truncated Schauder basis (up to a given number of terms)
 and then solves the transformed problem \eqref{svm0} in this new extended space of original variables. This provides the classification in the extended space and this classification is applied to the original data. By standard arguments based on continuity and compactness given a prespecified accuracy the truncation order can be fixed to ensure the result.

\begin{ex}\label{ex:4}
We illustrate the \textit{primal} methodology for the same  dataset of Example \ref{ex:3}. If the transformation provided in such an example is used to compute the $\ell_{\frac{3}{2}}$-SVM, we get the following primal formulation:
\begin{align*}
\min  \;\;\; &t +10 \xi_{1} + 10 \xi_{2} + 10 \xi_{3} + 10 \xi_{4} + 10 \xi_{5} + 10 \xi_{6}\\
\mbox{s.t. } &  b + \xi_{1} \geq 1,\\
 &  \omega_{3} + b + \xi_{2} \geq 1,\\
 &   \omega_{1} + b + \xi_{3} \geq 1,\\
 &  -\omega_{1} - \sqrt[3]{2}  \omega_{2} -  \omega_{3} - b + \xi_{4} \geq 1,\\
 &   -\omega_{1} + \sqrt[3]{2}  \omega_{2} -  \omega_{3} - b + \xi_{5} \geq 1,\\
 &   - \omega_{1} + \sqrt[3]{2}  \omega_{2} -  \omega_{3} - b + \xi_{6} \geq 1,\\
 & t^2 \geq \|\omega\|_{\frac{3}{2}}^3,\\
 &  \xi_i\geq 0, i=1, \ldots, 6,\\
 & b \in \R,  \omega_j \in \R, j=1,2,3.
 \end{align*}
Note that the constraint $t^2 \geq \|w\|_{\frac{3}{2}}^3$ can be equivalently rewritten, by introducing the auxiliary variables $\zeta_1, \zeta_2$ and $\zeta_3$, as:
\begin{equation*}
 \begin{dcases}
 t^2 \geq \sum_{i=1}^d u_j, \mbox{\hspace*{1.5cm}}\\
 v_j \geq  \omega_j,   v_j \geq - \omega_j,  j=1,2,3,\\
u_j \zeta_j \geq v_j^2, j=1,2,3,\\
v_j \geq \zeta_j^2, j=1,2,3,\\
\end{dcases}
\end{equation*}
since, for each $j=1,2,3$, $v_j$ represents $| \omega_j|$, and because of the above non-linear constraints, we have that:
$$
v_j^4 \leq \zeta^2_j u_j^2 \leq u_j^2 v_j \Rightarrow u_j^2  \geq v_j^3 \quad (u_j \geq v_j^{\frac{3}{2}})
 $$

 Thus, solving the above second order cone programming problem we get  $\omega^* = (2,0,2)$ and $b^*=3$.
We also obtain that all the misclassifying errors $\xi$ are equal to zero. The same result was obtained in Example \ref{ex:3}. In Figure \ref{fig}, we draw (left picture) the separating curve when projecting the obtained hyperplane onto the original feature space.

Let us consider now the Schauder basis for continuous functions that consists of all monomials in $\R[z_1, \ldots, z_d]$.  One may define the transformation $\Phi: \R^d \rightarrow \R^{\N^d}$, $\Phi_\gamma(\mathrm{z}) = \mathrm{z}^\gamma$, for each $\gamma \in \N^d$. Note that $\Phi$ \textit{projects} the original finite-dimensional feature space onto the infinite dimensional space of sequences $\{z^\gamma\}_{\gamma\in \mathbb{N}^d}$. Hence, for any $z \in \R^d$ and $\gamma \in \N^d$,  the $\gamma$ component of $\Phi$, $\Phi_\gamma(z)$, is a real number. Truncating the basis $\mathcal{B}$ by a given order $\eta \in \N$, we define the transformation $\Phi[\eta]: \R^d \rightarrow \R^{\N^d_\eta}$, $\Phi[\eta]_\gamma(\mathrm{z}) = \mathrm{z}^\gamma$, for each $\gamma \in \N^d_\eta$.  Note that $\R^{\N^d_\eta}$ is a finite-dimensional space with dimension $\bfrac{d+\eta}{d}$.

 For instance, using $\Phi[3]$, the data are transformed into:
\begin{align*}
 X^\prime = &\{(1, 0, 0, 0, 0, 0, 0, 0, 0, 0), (1, 0, 1, 0, 0, 1, 0, 0, 0, 1), (1, 1, 0, 1, 0, 0, 1, 0, 0, 0),\\
 & (1, 1, 1, 1, 1, 1, 1, 1, 1, 1), (1, 1, -1, 1, -1, 1, 1, -1, 1, -1), \\
& (1, -1, 1, 1, -1, 1, -1, 1, -1, 1)\} \subseteq \R^{10}
\end{align*}
Then, solving \eqref{svm0} for this new dataset, we get, that in this new feature space (of dimension $10$), the optimal coefficients are:
$$
\omega^*=(0,0.1117,0.1117,-1.3295,0.4469,-1.3295,0.1117,-0.6704,-0.6704,0.1117),
$$
$$
b^*=2.1060,
$$
which define, when projecting it onto the original feature space, the curve drawn in Figure \ref{fig} (center). This solution also perfectly classifies the given points.

If we truncate the Schauder basis up to degree $4$ using $\Phi[4]$ (transforming the data into a $15$-dimensional space), we obtain the curve drawn in the right side of Figure \ref{fig}.

\begin{figure}[H]
\centering\includegraphics[scale=0.3]{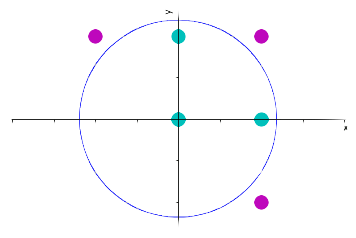} \includegraphics[scale=0.3]{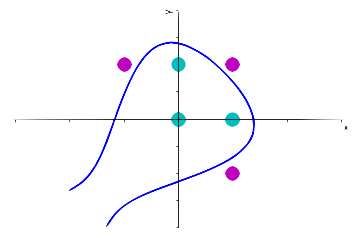} \includegraphics[scale=0.3]{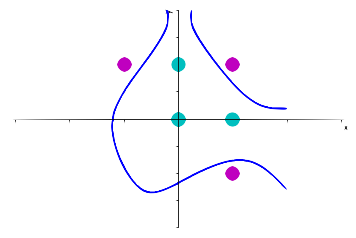}
\caption{Separating curves with the three different settings (left: using the quadratic transformation, center: using $\Phi[3]$, right: : using $\Phi[4]$).\label{fig}}
\end{figure}
\end{ex}

\section{Experiments}\label{sec:5}

We have performed a series of experiments to analyze the behavior of the proposed methods on real-world benchmark data sets. We have implemented the primal second-order cone formulation \eqref{cc1}--\eqref{cc5} and, in order to find non-linear separators, we consider the following two types of transformations on the data which can be identified  with adequate truncated Schauder bases:
\begin{itemize}
\item $\Phi[\eta]: \R^{d} \rightarrow \R^{\N^d_\eta}$. Its components, $\Phi[\eta]_\gamma(\mathrm{z})=  z^\gamma$ for $\gamma \in \N^d_\eta$, are the monomials (in $d$ variables) up to degree $\eta$.
\item  $\widetilde{\Phi}[\eta]: \R^d \rightarrow \R^{\N^d_\eta}$, with
$\widetilde{\Phi}[\eta]_\gamma(\mathbf{z}) = {\rm exp}({-\sigma \|z\|_2^2}) \dfrac{\sqrt[r]{2\sigma} \mathbf{z}^\gamma}{\sqrt[r]{\gamma_1! \cdots \gamma_d!}}$,  for $\mathbf{z} \in \R^d$, for $\gamma \in \N^d_\eta$ and $\sigma>0$.

\end{itemize}
Although both transformations have a similar shape (their components consist of monomials of certain degrees), the second one has non-unitary coefficients. Those coefficients come from the construction of the Gaussian transformation which turns out to be the Gaussian kernel. In  this second case,  the higher the order, the  closer the induced (polynomial) kernel to the gaussian kernel. Observe that the following generalized Gaussian operator $\mathbb{G}: \R^{r\times d} \rightarrow \R$ defined as
\begin{align*}
\mathbb{G}[\mathbf{x}_{i_1\cdot},\ldots,\mathbf{x}_{i_r\cdot}] &= {\rm exp}({-\sigma \dsum_{a,b=1}^r \|x_{i_a}-x_{i_b}\|_2^2}) ¡\\
&= {\rm exp}({-\sigma\dsum_{a=1}^r \|x_{i_a}\|_2^2}) \dsum_{\gamma \in \N^d}  \dfrac{2\sigma}{\gamma_1! \cdots \gamma_d!} \mathbf{x}_{i_1\cdot}^\gamma \cdots \mathbf{x}_{i_r\cdot}^\gamma, \end{align*}

is induced by using the transformation $\widetilde{\Phi}$, i.e., $\mathbb{G}[\mathbf{x}_{i_1\cdot},\ldots,\mathbf{x}_{i_r\cdot}] = \dlim_{\eta\to\infty} \sum_{\gamma \in \N^d_\eta} \widetilde{\Phi}[\eta]_\gamma$ where $\gamma=\sum_{l=1}^r  {\rm e}_{i_l}$. Hence, the transformation $\widetilde{\Phi}[\eta]$, for a given $\eta$, is nothing but a truncated expansion of the generalized Gaussian operator $\mathbb{G}$.

We construct, in our experiments, $\ell_p$-SVM separators for $p\in \{\frac{4}{3}, \frac{3}{2}, 2, 3\}$ by using $\eta$-order approximations with $\eta$ ranging in $\{1,2,3,4\}$. The case $\eta=1$ coincides with the linear separating hyperplane for both transformations.

The resulting primal  Second Order Cone  Programming  (SOCP) problem
 was coded in Python 3.6, and solved using Gurobi 7.51 in a Mac OSX El Capitan with an Intel Core i7
processor at 3.3 GHz and 16GB of RAM.

The models were tested in four classical data sets, widely used in the literature of SVM, that are listed in Table \ref{datasets}. They were obtained from the UCI Repository \cite{cleveland} and LIBSVM Datasets\cite{libsvm}. There, one can find further information about each of them.

\begin{table}[h]
\centering\begin{tabular}{r|ccc}\hline
Name & \# Obs. ($n$) & \# Features ($d$) & Source\\\hline
\texttt{cleveland} & 303 & 13 & UCI Repository\\
\texttt{housing} & 506 & 13 & UCI Repository\\
\texttt{german credit} &1000 &24 & UCI Repository\\
\texttt{colon}&62&2000& LIBSVM Datasets\\\hline
\end{tabular}
\caption{Datasets used in our experiments.\label{datasets}}
\end{table}

In order to obtain stable and meaningful results, we use a $10$-fold cross validation scheme to train the model and to test its performance. We report the accuracy of the model, which is defined as:
$$
{\rm ACC} = \dfrac{TP + TN}{n} \cdot 100
$$
where $TP$ and $TN$ are the true positive and true negative predicted values after applying the model built with the training data set to a dataset (in our case to the training or the test sample). ${\rm ACC}$ is actually the percentage of well-classified observations. We report both the averages ACC for the training data (${\rm ACC}^{\rm Tr}$) and the test data (${\rm ACC}^{\rm Test}$), and also the averages CPU times for solving each one of the ten fold cross validation subproblems using the training data. We also report the average percentage of nonzero coefficients of the optimal separating hyperplanes, over the total number of variables of the problem ($\%{\rm NonZ}$).
\setlength{\tabcolsep}{0.8pt}

\begin{table}[h]
\hspace*{-1.5cm}\centering{\scriptsize\begin{tabular}{|r|rrrr|rrrr|rrrr|rrrr|}\cline{2-17}
\multicolumn{1}{c|}{} &  \multicolumn{4}{c|}{$\ell_{\frac{4}{3}}$} &\multicolumn{4}{c|}{$\ell_{\frac{3}{2}}$} & \multicolumn{4}{c|}{$\ell_{2}$} & \multicolumn{4}{c|}{$\ell_{3}$} \\\hline
$\eta$&  ${\rm ACC}^{\rm Tr}$ & ${\rm ACC}^{\rm Test}$  & ${\rm Time}$ & ${\rm \%NonZ}$ & ${\rm ACC}^{\rm Tr}$ & ${\rm ACC}^{\rm Test}$  & ${\rm Time}$ & ${\rm \%NonZ}$ &${\rm ACC}^{\rm Tr}$ & ${\rm ACC}^{\rm Test}$ &${\rm Time}$& ${\rm \%NonZ}$ & ${\rm ACC}^{\rm Tr}$ & ${\rm ACC}^{\rm Test}$ & ${\rm Time}$ & ${\rm \%NonZ}$ \\\hline\hline
\multicolumn{1}{c|}{}& \multicolumn{16}{c|}{\texttt{cleveland} dataset ($C=4$)}\\\hline
1& 85.11\% & 82.84\% &0.01& 100\%& 85.11\% & 83.16\% &0.01& 100\% & 85.15\% & {\bf 83.48\%} &0.01& 100\% & 85.33\% & 83.15\% &0.01& 100\% \\
2& 94.02\% & {\bf 82.57\%} &0.44& 88.86\%& 93.58\% & 81.57\% &0.40& 94.48\% & 93.33\% & 81.58\% &0.04& 98.95\% & 93.35\% & 79.61\% &0.41& 98.31\% \\
3& 99.34\% & 74.93\% &5.49& 72.02\%& 99.41\% & 75.60\% &2.87& 84.84\% & 99.67\% & 78.53\% &0.14& 98.82\% & 99.67\% & {\bf 80.23\%} &2.65& 99.66\% \\
4& 99.67\% & 76.56\% &28& 72.00\%& 99.67\% & 76.92\% &22.5& 81.88\% & 99.74\% & {\bf 79.21\%} &0.47& 97.54\% & 100\% & 78.60\% &17.56& 99.31\% \\\hline
\multicolumn{1}{c|}{}& \multicolumn{16}{c|}{\texttt{housing} dataset ($C=64$)}\\\hline
1& 88.56\% & {\bf 85.36\%} &0.01& 100\%& 88.25\% & 85.16\% &0.02& 100\% & 88.10\% & 84.36\% &0.02& 100\% & 87.92\% & 83.35\% &0.04& 100\% \\
2& 94.93\% & 78.85\% &0.22& 90.57\%& 94.14\% & 80.03\% &0.42& 96.67\% & 92.31\% & 80.02\% &0.14& 99.05\% & 91.15\% & {\bf 81.38\%} &0.39& 98.86\% \\
3& 98.60\% & {\bf 80.95\%} &9.57& 57.36\%& 98.24\% & 80.00\% &6.13& 74.84\% & 97.34\% & 79.81\% &0.51& 97.27\% & 96.07\% & 78.84\% &5.86& 99.59\% \\
4& 99.23\% & {\bf 79.99\%} &45.09& 50.82\%& 98.90\% & 77.78\% &31.69& 68.32\% & 98.37\% & 78.63\% &1.59& 95.30\% & 97.98\% & 78.43\% &27.42& 98.53\% \\\hline
\multicolumn{1}{c|}{}& \multicolumn{16}{c|}{\texttt{german} credit dataset ($C=64$)}\\\hline
1& 78.53\% & {\bf 76.20\%} &0.02& 99.58\%& 78.53\% & {\bf 76.20\%} &0.04& 99.58\% & 78.53\% & {\bf 76.20\%} &0.05& 99.58\% & 78.54\% & {\bf 76.20\%} &0.04& 99.58\% \\
2& 93.03\% & 67.50\% &0.92& 96.62\%& 93.04\% & 67.60\% &2.50& 98.15\% & 92.98\% & 67.40\% &0.50& 99.69\% & 93.00\% & {\bf 67.70\%} &3.32& 99.75\% \\
3& 100\% & {\bf 71.90\%} &85.86& 60.93\%& 100\% & 70.50\% &94.12& 78.20\% & 100\% & 70.20\% &3.14& 96.76\% & 100\% & 68.90\% &98.58& 99.65\% \\\hline
\multicolumn{1}{c|}{}& \multicolumn{16}{c|}{\texttt{colon} dataset ($C=1$)}\\\hline
1& 100\% & {\bf 82.14\%} &20.3& 46.14\%& 100\% & 80.48\% &15.73& 64.54\% & 100\% & 80.48\% &0.05& 89.74\% & 100\% & 80.48\% &14.61& 99.44\%\\\hline
\end{tabular}
\caption{Results of our computational experiments for ${\Phi}[\eta]$.\label{table1}}}
\end{table}

Since our models depend on two parameters ($C$ and $\eta$) and one more  ($\sigma)$ in case of using the transformation $\widetilde{\Phi}[\eta]$, we first perform a test to find the best choices for $C$ and $\sigma$. For each dataset, we consider a part of the training sample and run the models by moving $C$ and $\sigma$ over the grid $\{2^{k}: k \in \{-7,-6, \ldots, 6, 7\}\}$. For each dataset, the best combination of parameters is identified and chosen. Then, it is used for the rest of the experiments on such a dataset.

\begin{table}[h]
\hspace*{-1.5cm}\centering{\scriptsize\begin{tabular}{|r|rrrr|rrrr|rrrr|rrrr|}\cline{2-17}
\multicolumn{1}{c|}{} &  \multicolumn{4}{c|}{$\ell_{\frac{4}{3}}$} &\multicolumn{4}{c|}{$\ell_{\frac{3}{2}}$} & \multicolumn{4}{c|}{$\ell_{2}$} & \multicolumn{4}{c|}{$\ell_{3}$} \\\hline$\eta$&  ${\rm ACC}^{\rm Tr}$ & ${\rm ACC}^{\rm Test}$  & ${\rm Time}$ & ${\rm \%NonZ}$ & ${\rm ACC}^{\rm Tr}$ & ${\rm ACC}^{\rm Test}$  & ${\rm Time}$ & ${\rm \%NonZ}$ &${\rm ACC}^{\rm Tr}$ & ${\rm ACC}^{\rm Test}$ &${\rm Time}$& ${\rm \%NonZ}$ & ${\rm ACC}^{\rm Tr}$ & ${\rm ACC}^{\rm Test}$ & ${\rm Time}$ & ${\rm \%NonZ}$ \\\hline\hline
\multicolumn{1}{c|}{}& \multicolumn{16}{c|}{\texttt{cleveland} dataset ($C=2$ and $\sigma=2^{-6}$)}\\\hline
1& 85.15\% & 83.16\% &0.01& 99.23\%& 85.11\% & 83.16\% &0.01& 99.23\% & 85.33\% & {\bf 83.48\%} &0.01& 100\% & 85.22\% & {\bf 83.48\%} &0.01& 100\% \\
2& 88.30\% & {\bf 84.19\%} &0.24& 66.86\%& 88.05\% & 82.55\% &0.28& 84.00\% & 86.72\% & 80.58\% &0.04& 99.52\% & 84.01\% & 77.26\% &0.24& 99.81\% \\
3& 92.15\% & 80.87\% &4.91& 49.25\%& 92.12\% & 81.54\% &2.77& 68.50\% & 92.41\% & {\bf 81.55\%} &0.13& 96.54\% & 92.59\% & 81.20\% &2.54& 99.68\% \\
4& 84.38\% & 83.47\% &19.57& 3.53\%& 84.41\% & 83.46\% &12.83& 8.47\% & 84.71\% & 83.46\% &0.19& 41.97\% & 85.18\% & {\bf 83.48\%} &15.51& 63.50\% \\\hline
\multicolumn{1}{c|}{}& \multicolumn{16}{c|}{\texttt{housing} dataset ($C=64$ and $\sigma=2^{-6}$)}\\\hline
1& 88.56\% & {\bf 85.36\%} &0.01& 100\%& 88.25\% & 85.16\% &0.02& 100\% & 88.10\% & 84.36\% &0.02& 100\% & 87.53\% & 84.71\% &0.04& 100\% \\
2& 89.53\% & {\bf 83.53\%} &0.25& 75.14\%& 88.84\% & 82.95\% &0.48& 88.48\% & 87.42\% & 82.94\% &0.11& 99.24\% & 86.72\% & 82.46\% &0.66& 100\% \\
3& 94.01\% & 80.03\% &4.47& 37.38\%& 93.30\% & 79.82\% &4.29& 54.52\% & 91.50\% & {\bf 80.21\%} &0.25& 88.30\% & 90.36\% & 79.95\% &3.05& 99.62\% \\
4& 90.80\% & 82.37\% &14.43& 4.23\%& 90.58\% & {\bf 83.36\%} &20.98& 7.56\% & 88.95\% & 81.59\% &0.17& 20.31\% & 86.69\% & 82.95\% &12.2& 65.97\% \\\hline
\multicolumn{1}{c|}{}& \multicolumn{16}{c|}{\texttt{german credit} dataset ($C=0.25$ and $\sigma=2^{-6}$)}\\\hline
1& 78.35\% & {\bf 79.00\%} &0.02& 99.48\% & 78.33\% & 78.88\% &0.04& 100\% & 78.25\% & 78.63\% &0.05& 100\% & 78.26\% & 78.75\% &0.04& 100\%\\
2& 77.29\% & 74.38\% &2.96& 90.23\%& 77.83\% & 75.00\% &2.37& 97.62\% & 79.23\% & 74.44\% &0.45& 99.97\% & 81.15\% & {\bf 75.22\%} &2.13& 100\% \\
3& 76.72\% & 76.75\% &57.01& 3.39\%& 92.78\% & {\bf 79.00\%} &63.64& 91.69\% & 96.36\% & 77.88\% &2.75& 99.82\% & 98.24\% & 76.57\% &48.4& 99.99\% \\\hline
\multicolumn{1}{c|}{}& \multicolumn{16}{c|}{\texttt{colon} dataset ($C=1$)}\\\hline
1& 100\% & {\bf 82.14\%} &20.3& 46.14\%& 100\% & 80.48\% &15.73& 64.54\% & 100\% & 80.48\% &0.05& 89.74\% & 100\% & 80.48\% &14.61& 99.44\%\\\hline
\end{tabular}
\caption{Results of our computational experiments for transformation $\widetilde{\Phi}[\eta]$.\label{table2}}}
\end{table}

Tables \ref{table1} and \ref{table2} report, respectively, the average results for the feature transformations $\Phi[\eta]$ and $\widetilde{\Phi}[\eta]$.
We report the results on those choices of $\eta$  that result in a good compromise between some improvement in accuracy and complexity on the problem solving. For instance, while for the datasets \texttt{cleveland} and \texttt{housing}, a degree up to $\eta=4$ was considered, for \texttt{german credit} a degree of $\eta=3$ already allows us to perfectly fit the data (${\rm ACC}^{\rm Tr}=100\%$), and for \texttt{colon}, $\eta=1$, i.e., the linear fitting, is enough to  correctly classify the training sample. The best accuracy results for each $\eta$ and each dataset are boldfaced. As a general observation of our  experiment  if $\Phi[\eta]$ is used, there is no gain (in terms of accuracy on the testing sample) by increasing the value of $\eta$ since the linear hyperplane is the one where we got the best results. However, such a situation changes when $\widetilde{\Phi}[\eta]$ is used since we found datasets (as \texttt{cleveland} or \texttt{german}) in which the best accuracy results are obtained for non-linear transformations. It can be also observed that a best fitting for the training data does not always imply the best performance for the test data. This behavior may be due to overfitting.

Concerning the use of different norms, one can observe that there is no a significative best one in terms of accuracy on the test sample, although we obtain most of the best results using $\ell_{\frac{4}{3}}$. At this point, we would like to remark that the usual norm used in SVM, the Euclidean norm, does not stand out over the others. On the other hand, the $\ell_2$-norm cases are solved in much smaller computational times than the other, since this norm is directly  representable as a  single quadratic constraint in our model, while the others need to consider auxiliary variables and constraints which increase the complexity for solving the problem. In spite of that, the remaining computational times are reasonable with respect to the size of the instances.

In terms of the number of features used in the hyperplane (those with nonzero optimal $\omega$-coefficients), the one which uses the less number of them is, as expected,  the $\ell_{\frac{4}{3}}$-norm since it is \textit{closest} to the $\ell_1$-norm which is known to be highly sparse.

\section{Conclusions}

The concept of classification margin is on the basis of the support vector machine technology to classify data sets. The measure of this margin has been usually done using Euclidean  ($\ell_2$) norm, although some alternative attempts can be found in the literature, mainly with  $\ell_1$ and $\ell_\infty$ norms. Here, we have addressed the analysis of  a general framework for support vector machines with the family of $\ell_p$-norms with $p>1$.  Based on the properties and geometry of the considered models and  norms we have derived a unifying theory that allows us to obtain new classifiers that subsume most of the previously considered cases as particular instances.  Primal and dual formulations for the problem are provided,  extending those already known in the literature. The dual formulation permits to extend the so-called \textit{kernel trick}, valid for the $\ell_2$-norm case, to more general cases with $\ell_p$-norms, $p>1$. The tools that have been used in our approach combine modern mathematical optimization and geometrical and  tensor analysis. Moreover, the contributions of this paper are not only theoretical but also computational: different solution approaches have been developed and tested on four standard benchmark datasets from the literature. In terms of separation and classification no clear domination exists among the different possibilities and models, although in many cases the use of the standard SVM with $\ell_2$-norm is improved by other norms (as for instance the $\ell_{4/3}$). Analyzing and comparing the different models may open new avenues for further research, as for instance the application to categorical data by introducing additional binary variables in our models as it has been recently done in the standard SVM model, see e.g.\cite{cam17}.

\section*{Acknowledgements}

The first and second authors were partially supported by the project MTM2016-74983-C2-1-R (MINECO, Spain). The third author was partially supported by the project  MTM2016-74983-C2-2-R (MINECO, Spain). The first author was also partially supported by the research project PP2016-PIP06 (Universidad
de Granada) and the research group SEJ-534 (Junta de Andaluc\'ia).


\begin{thebibliography}{99}
\bibitem{writing} C.~Bahlmann, B.~Haasdonk, and H.~Burkhardt (2002). \emph{On-Line Handwriting Recognition with Support Vector Machines: A Kernel Approach}. In Proceedings of the Eighth International Workshop on Frontiers in Handwriting Recognition (IWFHR'02).

\bibitem{BennettBredensteiner00} K.P.~Bennett and E.J.~Bredensteiner (2000). \emph{Duality and Geometry in SVM Classifiers}. ICML 2000: 57-64

\bibitem{BEP14} V.~Blanco, J.~Puerto, and S.~El~Haj Ben~Ali, \emph{Revisiting several problems
  and algorithms in continuous location with $\ell_\tau$-norms}, Computational
  Optimization and Applications \textbf{58} (2014), no.~3, 563--595.

\bibitem{Burges98} Ch.J.~Burges (1998). \emph{A Tutorial on Support Vector Machines for Pattern Recognition}. Data Min. Knowl. Discov. 2(2), 121-167.
\bibitem{cr13} E. Carrizosa and D. Romero-Morales (2013). \emph{Supervised classification and mathematical optimization}. Computers \& Operations Research, 40(1), 150--165.

\bibitem{cam17} E. Carrizosa, A. Nogales--G\'omez, D. Romero-Morales (2017). \emph{Clustering categories in support vector machines}.  Omega, 66, 28--37.
\bibitem{als} J. D. Carroll and J. J. Chang (1970). \emph{Analysis of individual differences in multidimensional scaling via an N-way generalization of Eckart-Young decomposition}, Psychometrika 35 , 283--319.
%

\bibitem{libsvm} C.C.~Chang and C.J.~Lin (2011). \emph{LIBSVM -- A Library for Support Vector Machines}. ACM Transactions on Intelligent Systems and Technology 2(3), 1--27. Available at \url{https://www.csie.ntu.edu.tw/~cjlin/libsvm/}

\bibitem{SOStensors} H.~Chen, G.~Li and L.~Qi (2016). \emph{SOS tensor decomposition: Theory and applications}. Communications in Mathematical Sciences 14 (8), 2073--2100.

\bibitem{comon:hal-00327599} P.~Comon, G.~Golub, L-H.~Lim, Lek-Heng \and B.~Mourrain (2008). \emph{Symmetric tensors and symmetric tensor rank}. SIAM Journal on Matrix Analysis and Applications 30(3), 1254--1279

\bibitem{CortesVapnik95} C.~Cortes and V.~Vapnik (1995). \emph{Support-Vector Networks}. Mach. Learn. 20(3), 273--297.

\bibitem{eckartand-young} C.~Eckart and G.~Young (1939). \emph{A principal axis transformation for non-Hermitian matrices}. Bull. Amer. Math. Soc. 4.
118--121.

\bibitem{Edelsbrunner87}
H.  Edelsbrunner (1987). \emph{Algorithms in combinatorial geometry}.  Springer-Verlag, Berlin.
	
\bibitem{GonzalezAbril11} L.~Gonzalez-Abril, F.~Velasco, J.A.~Ortega, and L.~Franco (2011). \emph{Support vector machines for classification of input vectors with different metrics}, Computers \& Mathematics with Applications 61(9), 2874--2878.

\bibitem{credit} T.~Harris (2013). \emph{Quantitative credit risk assessment using support vector machines: Broad versus Narrow default definitions}. Expert Syst. Appl. 40(11), 4404--4413.

\bibitem{gloptipoly} D.~Henrion, J. B.~Lasserre, and J.~Loefberg (2009). \emph{GloptiPoly 3: moments, optimization and semidefinite programming}. Optimization Methods and Software 24(4-5), 761--779.

\bibitem{hillar-lim2013} C. J.~Hillar and L.-H.~Lim (2013). \emph{Most tensor problems are NP-hard}. Journal of the ACM 60, 1--39.

\bibitem{asd} J. Jiang, H. Wu, Y. Li, and R.~Yu (2000). \emph{Three-way data resolution by alternating slice-wise
diagonalization (ASD) method}. Journal of Chemometrics 14, 15--36.

\bibitem{insurance} V.~Kascelan, L.~Kascelan, and M.~Novovic Buric (2016). \emph{A nonparametric data mining approach for risk prediction in car insurance: a case study from the Montenegrin market}. Economic Research-Ekonomska Istrazivanja 29(1), 545--558.

\bibitem{kofidis} E. Kofidis and P. A. Regalia (2002). \emph{On the best rank-1 approximation of higher-order supersymmetric tensors}. SIAM Journal on Matrix Analysis and Applications 23, 863--884.

\bibitem{IkedaMurata05a} K.~Ikeda and N.~Murata (2005). \emph{Geometrical Properties of Nu Support Vector Machines with Different Norms}. Neural Computation 17(11), 2508-2529.

\bibitem{IkedaMurata05b} K.~Ikeda and N.~Murata (2005). \emph{Effects of norms on learning properties of support vector machines}. ICASSP (5), 241-244


\bibitem{lasserrebook}
J.B~Lasserre (2009). \emph{Moments, Positive Polynomials and Their Applications}, Imperial College Press, London.


\bibitem{LiTza77} J.~Lindenstrauss and  L.~Tzafriri (1977). \emph{Classical Banach Spaces I, Sequence Spaces}. Ergebnisse der Mathematik und ihrer Grenzgebiete 92, Berlin: Springer-Verlag, ISBN 3-540-08072-4.


\bibitem{Liu07} Y.~Liu, H.H.~Zhang, C.~Park, and J.~Ahn (2007). \emph{Support vector machines with adaptive Lq penalty}. Comput. Stat. Data Anal. 51(12), 6380-6394.

\bibitem{cancer} A.~Majid, S.~Ali, M.~Iqbal, and N.~Kausar (2014). \emph{Prediction of human breast and colon cancers from imbalanced data using nearest neighbor and support vector machines}. Computer Methods and Programs in Biomedicine 113(3), 792--808.

\bibitem{mangasarian}  O.L.~Mangasarian (1999). \emph{Arbitrary-norm separating plane}. Oper. Res. Lett., 24 (1--
2):15--23.

\bibitem{mercer} J.~Mercer (1909). \emph{Functions of positive and negative type and their connection with the theory of integral equations}. Philosophical Transactions of the Royal Society A, 209, 415--446.

\bibitem{PedrosoMurata01} J.P.~Pedroso and N.~Murata (2001). \emph{Support vector machines with different norms: motivation, formulations and results}. Pattern Recognition Letters 22(12), 1263-1272.

\bibitem{putinar}
M.~Putinar (1993). \emph{Positive Polynomials on Compact Semi-Algebraic Sets}, Ind. Univ. Math. J. 42: 969-984.

\bibitem{Btensors} L.~Qi and Y.~Song (2014). \emph{An even order symmetric B tensor is positive definite}. Linear Algebra and its Applications 457, 303--312.


\bibitem{cleveland} S.~Radhimeenakshi (2016). \emph{Classification and prediction of heart disease risk using data mining techniques of Support Vector Machine and Artificial Neural Network}. 3rd International Conference on Computing for Sustainable Global Development (INDIACom), New Delhi, 3107--3111.

\bibitem{Cauchytensor} H.~Chen and L.~ Qi (2015). \emph{Positive definiteness and semi-definiteness of even order symmetric Cauchy tensors}. Journal of Industrial and Management Optimization 11(4), 1263--1274.

\bibitem{sdpt3} K.C.~Toh, M.J.~Todd, and R.H.~Tutuncu (1999). \emph{SDPT3 --- a Matlab software package for semidefinite programming}, Optimization Methods and Software 11, 545--581.

\bibitem{Vapnik95} V.N.~Vapnik (1995). \emph{The Nature of Statistical Learning Theory}. Springer-Verlag New York, Inc., New York, NY, USA.

\bibitem{Vapnik98} V.N.~Vapnik (1998). \emph{Statistical Learning Theory}. Wiley-Interscience.


\end{thebibliography}
\end{document}